\documentclass[10pt]{article}

\usepackage{comment,url,algorithm,algorithmic,graphicx,subcaption,relsize}
\usepackage{amssymb,amsfonts,amsmath,amsthm,amscd,dsfont,mathrsfs,mathtools,microtype,nicefrac,pifont}
\usepackage{float,psfrag,epsfig,color,url,hyperref}
\usepackage{upgreek}
\usepackage[dvipsnames]{xcolor}
\usepackage{epstopdf,bbm,mathtools,enumitem}
\usepackage[toc,page]{appendix}
\usepackage[mathscr]{euscript}
\usepackage[giveninits=true, style=alphabetic]{biblatex}

\usepackage[top=1in, bottom=1in, left=1in, right=1in]{geometry}

\def\balign#1\ealign{\begin{align}#1\end{align}}
\def\baligns#1\ealigns{\begin{align*}#1\end{align*}}
\def\balignat#1\ealign{\begin{alignat}#1\end{alignat}}
\def\balignats#1\ealigns{\begin{alignat*}#1\end{alignat*}}
\def\bitemize#1\eitemize{\begin{itemize}#1\end{itemize}}
\def\benumerate#1\eenumerate{\begin{enumerate}#1\end{enumerate}}

\newenvironment{talign*}
 {\csname align*\endcsname}
 {\endalign}
\newenvironment{talign}
 {\csname align\endcsname}
 {\endalign}

\def\balignst#1\ealignst{\begin{talign*}#1\end{talign*}}
\def\balignt#1\ealignt{\begin{talign}#1\end{talign}}

\let\originalleft\left
\let\originalright\right
\renewcommand{\left}{\mathopen{}\mathclose\bgroup\originalleft}
\renewcommand{\right}{\aftergroup\egroup\originalright}

\def\ind{\mathbbm 1}

\def\tinycitep*#1{{\tiny\citep*{#1}}}
\def\tinycitealt*#1{{\tiny\citealt*{#1}}}
\def\tinycite*#1{{\tiny\cite*{#1}}}
\def\smallcitep*#1{{\scriptsize\citep*{#1}}}
\def\smallcitealt*#1{{\scriptsize\citealt*{#1}}}
\def\smallcite*#1{{\scriptsize\cite*{#1}}}

\def\<{\left\langle} 
\def\>{\right\rangle}

\def\defeq{\triangleq} 

\newcommand{\inner}[1]{{\left\langle #1 \right\rangle}} 

\def\P{\mbb{P}} 

\DeclareSymbolFont{rsfs}{U}{rsfs}{m}{n}
\DeclareSymbolFontAlphabet{\mathscrsfs}{rsfs}

\newcommand{\eqdist}{\stackrel{d}{=}}

\ifdefined\nonewproofenvironments\else
\ifdefined\ispres\else
\newtheorem{theorem}{Theorem}
\newtheorem{lemma}[theorem]{Lemma}
\newtheorem{corollary}[theorem]{Corollary}
\newtheorem{definition}[theorem]{Definition}

\renewenvironment{proof}{\noindent\textbf{Proof.}\hspace*{.3em}}{\qed \vspace{.1in}}
\newenvironment{proof-sketch}{\noindent\textbf{Proof Sketch}
  \hspace*{1em}}{\qed\bigskip\\}
\newenvironment{proof-idea}{\noindent\textbf{Proof Idea}
  \hspace*{1em}}{\qed\bigskip\\}
\newenvironment{proof-of-lemma}[1][{}]{\noindent\textbf{Proof of Lemma {#1}}
  \hspace*{1em}}{\qed\\}
  \newenvironment{proof-of-proposition}[1][{}]{\noindent\textbf{Proof of Proposition {#1}}
  \hspace*{1em}}{\qed\\}
\newenvironment{proof-of-theorem}[1][{}]{\noindent\textbf{Proof of Theorem {#1}}
  \hspace*{1em}}{\qed\\}
\newenvironment{proof-attempt}{\noindent\textbf{Proof Attempt}
  \hspace*{1em}}{\qed\bigskip\\}

\newtheorem*{remark*}{Remark}

\fi

\newtheorem{proposition}[theorem]{Proposition}

\newtheorem{assumption}{Assumption}
\theoremstyle{definition}
\newtheorem{example}[theorem]{Example}
\fi
\makeatletter
\@addtoreset{equation}{section}
\makeatother

\hypersetup{
  colorlinks,
  linkcolor={red!50!black},
  citecolor={blue!50!black},
  urlcolor={blue!80!black}
}

\renewcommand{\Pr}[1]{\mathbb{P}\left( #1 \right)}

\definecolor{OliveGreen}{rgb}{0,0.6,0}

\definecolor{OliveGreen}{rgb}{0,0.6,0}
\usepackage{custom}
\usepackage[normalem]{ulem}

\makeatletter
\renewcommand{\paragraph}{%
  \@startsection{paragraph}{4}%
  {\z@}{1.25ex \@plus 1ex \@minus .2ex}{-1em}%
  {\normalfont\normalsize\bfseries}%
}
\makeatother

\begin{document}

\title{Improved Discretization Analysis for Underdamped\\ Langevin Monte Carlo}

 \author{Matthew Zhang\thanks{
  Department of Computer Science at
  University of Toronto, and Vector Institute, \texttt{matthew.zhang@mail.utoronto.ca}
} \ \ \ \ \ \ 
 Sinho Chewi\thanks{
  Department of Mathematics at
  Massachusetts Institute of Technology, \texttt{schewi@mit.edu}
 } \ \ \ \ \ \
 Mufan (Bill) Li\thanks{
  Department of Statistical Sciences at
  University of Toronto, and Vector Institute, \texttt{mufan.li@mail.utoronto.ca}
} \\
 Krishnakumar Balasubramanian\thanks{
 Department of Statistics at University of California, Davis, \texttt{kbala@ucdavis.edu}
 } \ \
 Murat A.\ Erdogdu\thanks{
  Department of Computer Science at
  University of Toronto, and Vector Institute, \texttt{erdogdu@cs.toronto.edu}
 }
}

\maketitle

\begin{abstract}

Underdamped Langevin Monte Carlo (ULMC) is an algorithm used to sample from unnormalized densities by leveraging the momentum of a particle moving in a potential well. 
We provide a novel analysis of ULMC, motivated by two central questions: (1)~\emph{Can we obtain improved sampling guarantees beyond strong log-concavity?} (2)~\emph{Can we achieve acceleration for sampling?}

For (1), prior results for ULMC only hold under a log-Sobolev inequality together with a restrictive Hessian smoothness condition. 
Here, we relax these assumptions by removing the Hessian smoothness condition and by considering distributions satisfying a Poincar\'e inequality.
Our analysis achieves the state of art dimension dependence, and is also flexible enough to handle weakly smooth potentials. 
As a byproduct, we also obtain the first KL divergence guarantees for ULMC without Hessian smoothness under strong log-concavity, which is based on a new result on the log-Sobolev constant along the underdamped Langevin diffusion.

For (2), the recent breakthrough of Cao, Lu, and Wang (2020) established the first accelerated result for sampling in continuous time via PDE methods. Our discretization analysis translates their result into an algorithmic guarantee, which indeed enjoys better condition number dependence than prior works on ULMC, although we leave open the question of full acceleration in discrete time.

Both (1) and (2) necessitate R\'enyi discretization bounds, which are more challenging than the typically used Wasserstein coupling arguments.  
We address this using a flexible discretization analysis based on Girsanov's theorem that easily extends to more general settings. 

\end{abstract}

\section{Introduction}
The problem of sampling from a high-dimensional distribution $\pi \propto \exp(-U)$ on $\mbb{R}^d$, 
 when the normalizing constant is unknown and only the potential $U$ is given, 
 has increasing relevancy in a number of application domains, including economics, physics, and scientific computing \cite{johannes2010mcmc, von2011bayesian, kobyzev2020normalizing}.
Recent progress on this problem has been driven by a strong connection with the field of optimization, starting from the seminal work of~\cite{jordan1998variational}; see~\cite{chewisamplingbook} for an exposition.

Given the success of momentum-based algorithms for optimization~\cite{nesterov1983method}, it is natural to investigate momentum-based algorithms for sampling. 
The hope is that such methods can improve the dependence of the convergence estimates on key problem parameters, such as the condition number $\kappa$, the dimension $d$, and the error tolerance $\epsilon$.
One such method is underdamped Langevin Monte Carlo (ULMC), which is a discretization of the underdamped Langevin diffusion (ULD):
\begin{align}
\label{eq:ULD}\tag{ULD}
\begin{aligned}
    &\D x_t = v_t \, \D t \,, \\
    &\D v_t = -\gamma v_t \, \D t - \nabla U(x_t) \, \D t +  \sqrt{2 \gamma} \, \D B_t \,, \nonumber
\end{aligned}
\end{align}
where $\{B_t\}_{t\geq 0}$ is the standard $d$-dimensional Brownian motion. 
The stationary distribution of ULD is $\mu(x,v) \propto \exp(-U(x) - \norm v^2/2)$, and in particular, the $x$-marginal of $\mu$ is the desired target distribution $\pi$.
Therefore, by taking a small step size for the discretization and a large number of iterations, ULMC will yield an approximate sample from $\pi$. 

We also note that in the limiting case where $\gamma = 0$, ULMC closely resembles the Hamiltonian Monte Carlo algorithm, which is known to achieve acceleration and better discretization error in some limited settings \cite{vishnoi2021introduction,apers2022hamiltonian,bou2022unadjusted, wang2022accelerating}. 

While there is currently no analysis of ULMC that yields acceleration for sampling (i.e., square root dependence on the condition number $\kappa$), ULMC is known to improve the dependence on other parameters such as the dimension $d$ and the error tolerance $\epsilon$~\cite{cheng2018sharp,cheng2018underdamped, dalalyan2020sampling}, at least for guarantees in the Wasserstein metric.
However, compared to the extensive literature on the simpler (overdamped) Langevin Monte Carlo (LMC) algorithm, existing analyses of ULMC are not easily extended to stronger performance metrics such as the $\msf{KL}$ and R\'enyi divergences. 
In turn, this limits the scope of the results for ULMC; see the discussion in Section~\ref{scn:contributions}.

In light of these shortcomings, in this work, we ask the following two questions: 
\vspace{-0.5em}
\begin{enumerate}[itemsep=0em]
    \item \emph{Can we obtain sampling guarantees beyond the strongly log-concave case via ULMC?}
    \item \emph{Can we obtain accelerated convergence guarantees for sampling via ULMC?}
\end{enumerate}

\subsection{Our Contributions}\label{scn:contributions}

We address the two questions above by providing a new Girsanov discretization bound for ULMC\@. Our bound holds in the strong R\'enyi divergence metric and applies under general assumptions (in particular, it does not require strong log-concavity of the target $\pi$, and it allows for weakly smooth potentials). Consequently, it leads to the following new state-of-the-art results for ULMC\@:

\vspace{-0.5em}
\begin{itemize}[itemsep=0em]
    \item We obtain an $\epsilon^2$-guarantee in KL divergence with iteration complexity $\widetilde{\mc{O}}(\kappa^{3/2} d^{1/2}\epsilon^{-1})$ for strongly log-concave and log-smooth distributions, which removes the Lipschitz Hessian assumption of~\cite{maetal2021nesterovmcmc}; here, $\kappa$ is the condition number of the distribution.
    \item We obtain an $\epsilon$-guarantee in TV distance with iteration complexity $\widetilde{\mc{O}}(C_{\msf{LSI}}^{3/2} L^{3/2} d^{1/2} \epsilon^{-1})$ under a log-Sobolev inequality (LSI) and $L$-smooth potential, again without assuming a Lipschitz Hessian.
    This is the state-of-the-art guarantee for this class of distributions with regards to dimension dependence.
    \item We obtain $\epsilon^2$-guarantees in the stronger R\'enyi divergence metric of any order in $[1, 2)$ with iteration complexity $\widetilde{\mc O}(C_{\msf{PI}}^{3/2} L^{3/2} \,d^2\epsilon^{-1})$ under a Poincar\'e inequality and a $L$-smooth potential, which improves to
    $\widetilde{\mc{O}}(C_{\msf{PI}} L d^{2}\epsilon^{-1})$ under log-concavity.
    These are the first guarantees for ULMC known in these settings, and they substantially improve upon the corresponding results for LMC in these settings~\cite{chewi2021analysis}.

    \item In the Poincar\'e case, we also consider weakly smooth potentials (i.e., H\"older continuous gradients with coefficient $s\in(0,1]$), which more realistically reflect the delicate smoothness properties of distributions satisfying a Poincar\'e inequality. 
\end{itemize}

We now discuss our results in more detail in the context of the existing literature.

\paragraph{Guarantees under Weaker Assumptions.}
Prior works,~\cite{cheng2018underdamped, dalalyan2020sampling, ganeshtalwar2020renyi}, require strong log-concavity of the target. 
Whereas for works which operate under isoperimetric assumptions, we are only aware of~\cite{maetal2021nesterovmcmc}, which further assumes a restrictive Lipschitz Hessian condition for the potential.
In contrast, we make no such assumption on the Hessian of $U$, and we obtain results under a log-Sobolev inequality (LSI), or under the even weaker assumption of a Poincar\'e inequality (PI), for which sampling analysis is known to be challenging~\cite{chewi2021analysis}.

As noted above, our result for sampling from distributions satisfying LSI and smoothness assumptions are state-of-the-art with regards to the dimension dependence ($d^{1/2}$); in contrast, the previous best results had linear dependence on $d$~\cite{chewi2021analysis, chen2022improved}.
Moreover, in the Poincar\'e case, we can also consider weakly smooth potentials, which have not been previously considered in the context of ULMC\@.

\paragraph{Guarantees in Stronger Metrics.}
Key to achieving these results is our discretization analysis in the R\'enyi divergence metric.
Indeed, the continuous-time convergence results for ULD under LSI or PI hold in the KL or R\'enyi divergence metrics, and translating these guarantees to the ULMC algorithm necessitates studying the discretization in R\'enyi. This is the main technical challenge, as we can no longer rely on Wasserstein coupling arguments which are standard in the literature~\cite{cheng2018underdamped, dalalyan2020sampling}.
Two notable exceptions are the R\'enyi discretization argument of~\cite{ganeshtalwar2020renyi}, which incurs suboptimal dependence on $\varepsilon$, and the KL divergence argument of~\cite{maetal2021nesterovmcmc}, which requires stringent smoothness assumptions.

In this work, we provide the first KL divergence guarantee for sampling from strongly log-concave and log-smooth distributions via ULMC without Hessian smoothness, based on a new LSI along the trajectory (discussed further below).

\paragraph{Towards Acceleration in Sampling.}
Our work is also motivated by the breakthrough result of~\cite{caoluwang2020underdamped}, which achieves for the first time an accelerated convergence guarantee for ULD in continuous time. Our discretization bound allows us to convert this result into an algorithmic guarantee which indeed improves the dependence on the condition number $\kappa$\footnote{In the case of Poincar\'e inequality, the condition number is $\kappa \deq C_{\msf{PI}} L$, which is consistent with the definition in the strongly log-concave case.} 
for ULMC, whereas prior results incurred a dependence of at least $\kappa^{3/2}$; our dependence is linear in $\kappa$ in the log-concave case.
While this still falls short of proving full acceleration for sampling (i.e., an improvement to $\kappa^{1/2}$), our result provides further hope for achieving acceleration via ULMC\@.

\paragraph{A New Log-Sobolev Inequality along the ULD Trajectory\@.} 
Finally, en route to proving the KL divergence guarantee in the strongly log-concave case, we establish a new log-Sobolev inequality along ULD (Proposition~\ref{lem:lsi_ct}), which is of independent interest.
While such a result was previously known for the overdamped Langevin diffusion, to the best of our knowledge it is new for the underdamped version.

\subsection{More Related Work}
\label{subsec:related_work}

\paragraph{Langevin Monte Carlo.}
For the standard LMC algorithm, non-asymptotic rate estimates in $\mc{W}_2$ were first demonstrated in \cite{dalalyan2017analogy}, for the class of strongly log-concave measures. 
Guarantees in $\msf{KL}$ divergence under a log-Sobolev inequality were obtained by \cite{vempalawibisono2019ula}, which developed an appealing continuous-time framework for analyzing LMC under functional inequalities. 
With some difficulty, this result was extended to R\'enyi divergences by \cite{ganeshtalwar2020renyi, erdhoszha22chisq}.
At the same time, a body of literature studied convergence in $\msf{KL}$ divergence under tail-growth conditions such as dissipativity \cite{raginskyrakhlintelgarsky2017sgld,erdogdu2018global,erdogduhosseinzadeh2021tailgrowth, mouetal22langevin}, which usually imply functional inequalities. 

Most related to the current work, \cite{chewi2021analysis} extended the continuous-time approach from \cite{vempalawibisono2019ula} to R\'enyi divergences, and moreover introduced a novel discretization analysis using Girsanov's theorem, which also holds for weakly smooth potentials.
The present work builds upon the Girsanov techniques introduced in \cite{chewi2021analysis} to study ULMC\@.

\paragraph{Underdamped Langevin Diffusion.} 
ULMC is a discretization of the underdamped Langevin diffusion~\eqref{eq:ULD}. 
First studied by \cite{kolmogoroff1934zufallige} and \cite{hormander1967hypoelliptic} in their pioneering works on hypoellipticity, 
it was quickly understood that establishing quantitative convergence to stationarity is technically challenging, let alone capturing any acceleration phenomenon. 
The seminal work of \cite{villani2002limites, villani2009hypocoercivity} developed the hypocoercivity approach, providing the first convergence guarantees under functional inequalities;
see also~\cite{herau2006hypocoercivity, dolbeault2009hypocoercivity, dolbeault2015hypocoercivity, roussel2018spectral}. We also refer to~\cite{bernard2022hypocoercivity} and references therein for a comprehensive discussion of qualitative and quantitative convergence results for ULD.

As mentioned earlier, the most recent breakthrough
by \cite{caoluwang2020underdamped} achieved acceleration in continuous time in $\chi_2$-divergence when the target distribution $\pi$ is log-concave. 
This work was built on an approach using the dual Sobolev space $\mc H^{-1}$ \cite{albritton2019variational}. 
However, since this method relies on the duality of the $L^2$ space and its connections to the Poincar\'e inequality, it is difficult to extend to $L^p$ spaces or to other functional inequalities. 

\paragraph{Other Discretizations.}
Many alternative discretization schemes have since been proposed in this setting \cite{shenlee2019randomizedmidpoint, li2019stochastic,hebalasubramanianerdogdu2020rm,foster2021shifted,monmarche2021high, foster2022high,johnston2023kinetic}, albeit all of the analyses up to this point were limited to $\mc W_2$ distance and did not achieve acceleration in terms of the condition number $\kappa$. 

\subsection{Organization}

The remainder of this paper will be organized as follows. 
In Section \ref{sec:background}, we will review the required definitions and assumptions. 
In Section \ref{sec:theorems}, we will state our main results and briefly sketch their proofs. 
In Section \ref{sec:applications}, we highlight several implications of our theorems through some examples. 
In Section \ref{sec:proof_sketch}, we briefly sketch the proofs of our main results, before concluding in Section \ref{sec:conclusion} with a discussion of future directions.

\section{Background}\label{sec:background}
\subsection{Notation}
Hereafter, we will use $\norm{\cdot}$ to denote the $2$-norm on vectors.
In general, we will only work with measures that admit densities on $\R^d$, and we will abuse notation slightly to conflate a measure with its density for convenience. 
The notation $a = \mc{O}(b)$ signifies that there exists an absolute constant $C>0$ such that $a \leq Cb$, and $\widetilde{\mc{O}}(\cdot)$ hides logarithmic factors. 
Similarly we write $a = \Theta(b)$ if there exist constants $c,C>0$ such that $cb \leq a \leq Cb$, and $\widetilde{\Theta}(\cdot)$ hides logarithmic factors. 
The stationary measure (in the position coordinate) is $\pi \propto \exp(-U)$, and $U$ will be referred to as the potential. 
We will use $L^2(\pi)$ to denote test functions $f$ where $\mathbb{E}_\pi \, f^2 < \infty$, and $\mc H^1(\pi)$ to denote weakly differentiable $L^2(\pi)$ functions where $\partial_{x_i} f \in L^2(\pi)$. 
Finally, the notations $\lesssim$, $\gtrsim$, $\asymp$ represent $\leq$, $\geq$, $=$ up to absolute constants.
Further notations are introduced in subsequent sections.

\subsection{Definitions and Assumptions}

In this subsection, we will define the relevant processes, divergences, and isoperimetric inequalities.
Firstly, we define the ULMC algorithm by the following stochastic differential equation (SDE): 
\begin{align}\label{eq:ulmc_sde}\tag{ULMC}
\begin{aligned}
    \D x_{t} &= v_t \, \D t\,, \\
    \D v_{t} &= - \gamma v_t \, \D t + \nabla U(x_{kh}) \, \D t + \sqrt{2\gamma} \, \D B_t\,, 
    \end{aligned}
\end{align}
where $t \in [kh, (k+1)h)$ for some step size $h>0$. 
We note this formulation of ULMC can be integrated in closed form (see Appendix \ref{sec:explicit_form}).

Next, we define a few measures of distance between two probability distributions $\mu$ and $\pi$ on $\mathbb{R}^d$. We define the total variation distance as
\begin{align}
    \norm{\mu-\pi}_{\msf{TV}} 
    \coloneqq \sup|\mu(A) - \pi(A)| 
    \,, 
\end{align}
where the $\sup$ is taken over Borel measurable sets $A \subset \mathbb{R}^d$. 
We further define the $\msf{KL}$ divergence as
\begin{align}
    \msf{KL}(\mu \mmid \pi ) \coloneqq \int \frac{\D \mu}{\D \pi} \log \frac{\D \mu}{\D \pi} \, \D \pi \,, 
\end{align}
and $\msf{KL}(\mu \mmid \pi ) \coloneqq +\infty$ if $\mu$ is not absolutely continuous with respect to $\pi$. 
Finally, we define the R\'enyi divergence with order $q > 1$ as
\begin{align*}
    \eu R_q(\mu \mmid \pi) \deq \frac{1}{q-1} \log \int \Bigl |\frac{\D \mu}{\D \pi} \Bigr|^q \, \D \pi\,,
\end{align*}
and similarly $\eu R_q(\mu \mmid \pi) \coloneqq +\infty$ if $\mu \not\ll \pi$. 
The R\'enyi divergence upper bounds $\msf{KL}$ for all orders, i.e., $\msf{KL}(\mu \mmid \pi) \leq \eu R_q(\mu \mmid \pi)$ for any order $q > 1$, and $\eu R_q$ is monotonic in $q$. 
In particular, when $q=2$, we also get $\chi_2$ divergence, i.e., $\chi_2(\mu \mmid \pi) =  \exp(\eu R_2(\mu \mmid \pi) )-1$. 

Our primary results are provided under the following smoothness conditions.
\begin{definition}[Smoothness]\label{as:smooth}
The potential $U$ is $(L,s)$-\textbf{weakly smooth} if $U$ is differentiable and $\nabla U$ is $s$-H\"older continuous satisfying 
\begin{align} \label{eq:holder}
    \norm{\nabla U(x) - \nabla U(y)} \leq L\, \norm{x-y}^s\,,
\end{align}
for all $x, y \in \mathbb{R}^d$ and some $L \geq 0$, $s\in (0,1]$. In the particular case where $s=1$, we say that the potential is $L$-smooth, or that $\nabla U$ is $L$-Lipschitz.
\end{definition}

We conduct three lines of analysis. The first assumes strong convexity of the potential, i.e.:
\begin{definition}[Strong Convexity] \label{def:str_cvx}
    The potential $U$ is $m$-\textbf{strongly convex} for some $m \geq 0$ if for all $x, y \in \R^d$:
    \begin{align*}
        \inner{\nabla U(x) - \nabla U(y), x-y} \geq \frac{m}{2}\, \norm{x-y}^2\,.
    \end{align*}
\end{definition}
In the case $m = 0$ above, we say that $U$ is convex. If a potential function $U$ is (strongly) convex, then we say the distribution $\pi \propto \exp(-U)$ is (strongly) log-concave.

A second, strictly more general assumption is the log-Sobolev inequality. 
\begin{definition}[Log-Sobolev Inequality]\label{def:lsi}
A measure $\pi$ satisfies a \textbf{log-Sobolev inequality} (LSI) with parameter $C_{\msf{LSI}} > 0$ if for all $g\in \mc H^1(\pi)$ :
\begin{align}\label{eq:LSI} \tag{LSI}
    \msf{ent}_{\pi}(g^2) \leq 2C_{\msf{LSI}} \E_{\pi}[\norm{\nabla g}^2]\,,
\end{align}
where $\msf{ent}_\pi(g^2) \deq \E_\pi[g^2 \log (g^2/\E_\pi[g^2])]$. 
\end{definition}
An $m$-strongly convex potential is known to satisfy \eqref{eq:LSI} with constant $m^{-1}$ \cite{bakrygentilledoux2014}. More generally, we can consider the following weaker isoperimetric inequality, which corresponds to a linearization of \eqref{eq:LSI}.
\begin{definition}[Poincar\'e Inequality]
    A measure $\pi$ satisfies a \textbf{Poincar\'e inequality} with parameter $C_{\msf{PI}} > 0$ if for all $g \in \mc H^1(\pi)$ :
\begin{align}\label{eq:pi}\tag{PI}
    \msf{var}_\pi(g) &\leq C_{\msf{PI}} \E_\pi [\norm{\nabla g}^2]\,,
\end{align}
where $\msf{var}_{\pi}(g) = \E_\pi[\abs{g-\E_\pi[g]}^2]$.
\end{definition}
Conditions \eqref{eq:LSI} and \eqref{eq:pi} are standard assumptions made on the stationary distribution in the theory of Markov diffusions as well as sampling \cite{bakrygentilledoux2014, vempalawibisono2019ula, chewi2021analysis, chewisamplingbook}. They are known to be satisfied by a broad class of targets such as log-concave distributions or certain mixture distributions \cite{chen2021kls, chenchewinilesweed2021dimfreelsi}.

We define the condition number for an $m$-strongly log-concave target with $(L,s)$-weakly smooth potential as $\kappa \defeq L/m$. In the case where instead of strong convexity, the target only satisfies \eqref{eq:LSI} (respectively \eqref{eq:pi}), the condition number is instead $\kappa \defeq C_{\msf{LSI}} L$ (respectively $\kappa \defeq C_{\msf{PI}} L$).

Finally, we collect several mild assumptions to simplify computing the bounds below, which have also appeared in prior work; see in particular the discussion in \cite[Appendix A]{chewi2021analysis}.
\begin{assumption}\label{as:misc}
The expectation of the norm (in the position coordinate) is quantitatively bounded by some constant, $\E_\pi[\norm{\cdot}] \leq \mf{m} = \widetilde{\mc{O}}(d)$\footnote{This holds for instance when $U(x) = \norm{x}^\alpha$ for $1 \le \alpha \le 2$.}, for some constant $\mf m < \infty$. 
Furthermore, we assume that $\nabla U(0) = 0$ (without loss of generality), and that  $U(0) - \min U = \widetilde{\mc O}(d)$.
\end{assumption}

\section{Main Theorems}\label{sec:theorems}

In the sequel, we always take the initial distribution of the  momentum $\rho_0$ to be equal to the stationary distribution $\rho \propto \exp(-\norm{\cdot}^2/2)$. Then, under Assumption \ref{as:misc} we can find an initial distribution $\pi_0$ for the position which is a centered Gaussian with variance specified in Appendix \ref{scn:proofs}, such that $\pi_0$ has some appropriately bounded initial divergence (e.g. $\msf{KL}, \eu R_q$) with respect to $\pi$.
Lastly, we initialize ULMC by sampling from the distribution $\mu_0(x,v) = \pi_0(x) \times \rho_0(v)$, i.e. with $x$ and $v$ independent.

\subsection{Convergence in \texorpdfstring{$\msf{KL}$}{KL} and \texorpdfstring{$\msf{TV}$}{TV}}
In order to state our results for ULMC in $\msf{KL}$ and $\msf{TV}$, we leverage the following result in continuous-time from \cite{maetal2021nesterovmcmc}, which relies on an entropic hypocoercivity argument, after a time-change of the coordinates (see Appendix~\ref{scn:entropic_hypo} for a proof).
\begin{lemma}[{Adapted from~\cite[Proposition 1]{maetal2021nesterovmcmc}}]\label{lem:lyapunov_decay}
    Define the Lyapunov functional
    \begin{align}\label{eq:F_q}
        \eu F(\mu' \mmid \mu) \defeq \msf{KL}(\mu' \mmid \mu) + \E_{\mu'}\bigl[\bigl\lVert \mf M^{1/2}\, \nabla \log \frac{\mu'}{\mu} \bigr\rVert^2 \bigr]\,, \  \quad \text{where} \ \ \mf M = \begin{bmatrix} \frac{1}{4L} & \frac{1}{\sqrt{2L}} \\ \frac{1}{\sqrt{2L}} & 4 \end{bmatrix}
        \otimes I_d
        \,.
    \end{align}
    For targets $\pi$ that are $L$-smooth and satisfy \eqref{eq:LSI} with parameter $C_{\msf{LSI}}$, let $\gamma = 2\sqrt{2L}$. Then the law $\mu_t$ of ULD satisfies
    \begin{align*}
        \partial_t \eu F(\mu_t \mmid \mu) \leq -\frac{1}{10C_{\msf{LSI}}\,\sqrt{2L}}\, \eu F(\mu_t \mmid \mu)\,.
    \end{align*}
\end{lemma}

We now proceed to state our main results more precisely. First, we obtain the following KL divergence guarantee under strong log-concavity and smoothness.

\begin{theorem}[{Convergence in $\msf{KL}$ under Strong Log-Concavity}]\label{thm:kl_slc}
    Let the potential $U$ be $m$-strongly convex and $L$-smooth, and additionally satisfy Assumption \ref{as:misc}. Then, for 
    \begin{align*}
    h = \widetilde{\Theta}\Bigl(\frac{\epsilon m^{1/2}}{Ld^{1/2}}\Bigl)\quad\text{and}\quad  \gamma \asymp \sqrt{L},
    \end{align*}
   the following holds for $\hat \mu_{Nh}$, the law of the $N$-th iterate of ULMC initialized at a centered Gaussian (with variance specified in Appendix~\ref{scn:proofs}):
    \begin{align*}
        \msf{KL}(\hat \mu_{Nh} \mmid \mu) \leq \epsilon^2 \qquad \text{after} \qquad N = \widetilde{\Theta} \Bigl(\frac{\kappa^{3/2}\, d^{1/2}}{\epsilon}\Bigr)\quad \text{iterations}\,.
    \end{align*}
\end{theorem}

Here, we justify the choice of error tolerance for $\msf{KL}$ to be $\epsilon^2$. 
Based on Pinsker's and Talagrand's transport inequalities, we know $\msf{KL}$ is on the order of $\msf{TV}^2, \mc{W}_2^2$. 
Hence, this allows for a fair comparison of convergence guarantees in terms of $\msf{KL}$ with $\msf{TV}$ and $\mc{W}_2$. 
Weakening the strong convexity assumption to $\eqref{eq:LSI}$, we obtain a result in $\msf{TV}$.

\begin{theorem}[{Convergence in $\msf{TV}$ under \eqref{eq:LSI}}]\label{thm:tv_lsi}
Let the potential be $L$-smooth, satisfy \eqref{eq:LSI} with constant $C_{\msf{LSI}}$, and satisfy Assumption \ref{as:misc}.
Then, for 
\begin{align*}
h = \widetilde{\Theta}\Bigl(\frac{\epsilon}{C_{\msf{LSI}}^{1/2} L d^{1/2}}\Bigr),\quad\text{and}\quad\gamma \asymp \sqrt{L},
\end{align*}
the following holds for $\hat \mu_{Nh}$, the law of the $N$-th iterate of ULMC initialized at a centered Gaussian (with variance specified in Appendix~\ref{scn:proofs}):
    \begin{align*}
        \norm{\hat \mu_{Nh} - \mu}_{\msf{TV}} \leq \epsilon \qquad\text{after} \qquad N = \widetilde{\Theta} \Bigl(\frac{C_{\msf{LSI}}^{3/2}\, L^{3/2}\,d^{1/2}}{\epsilon} \Bigr) \quad \text{iterations}\,.
    \end{align*}
    
\end{theorem}

\subsection{Convergence in \texorpdfstring{$\eu R_q$}{R\'enyi} and Improving the Dependence on \texorpdfstring{$\kappa$}{the Condition Number}}
To state our convergence results in $\eu R_q$, we additionally inherit the following technical assumption from \cite{caoluwang2020underdamped}.
\begin{assumption}\label{as:l2_asmpt}
$\mc H^1(\mu) \hookrightarrow L^2(\mu)$ is a compact embedding.
Secondly, assume that $U$ is twice continuously differentiable, and that for all $x \in \R^d$, we have
\begin{align*}
    \norm{\nabla^2 U(x)} \leq \mf L\, \left(1+\norm{\nabla U(x)} \right)\,.
\end{align*}
\end{assumption}
\begin{remark*}
\vspace{-0.2em}
    {\cite[Theorem 3.1]{hooton1981compact} shows the first part of this assumption is always satisfied if the potential has super-linear tail growth, i.e. $U(x) \propto \|x\|^\alpha$ for $\alpha>1$ and large $\|x\|$. 
    In the case where the tail is strictly linear, we can instead construct an arbitrarily close approximation with super-linear tails; thus, it generically holds for all targets we consider in this work. As also remarked in \cite{caoluwang2020underdamped}, the above assumption is required solely due to technical reasons and is likely not a necessary condition.
    }
    
    {
    The second part of the assumption is satisfied under $L$-smoothness of the gradient with the same constant. In the convex case or the case where $\nabla^2 U$ is lower bounded, the constant $\mf L$ does not show up in the bounds. As a result, for weakly smooth potentials in this setting, we can approximate using twice differentiable potentials to obtain a rate estimate.
    }
\end{remark*}

In the light of the above discussion, we emphasize that this additional assumption largely does not hinder the applicability of our results. 
Under this assumption, \cite{caoluwang2020underdamped} established the following guarantee on \eqref{eq:ULD} in continuous time. 

\begin{lemma}[{Rapid Convergence in $L^2$; Adapted from \cite[Theorem 1]{caoluwang2020underdamped}}]\label{lem:cont_time_pi} 
Under Assumption~\ref{as:l2_asmpt}, and if $\pi$ additionally satisfies \eqref{eq:pi} with constant $C_{\msf{PI}}$, then the following holds for the law $\mu_t$ of ULD initialized at $\mu_0$, where $C_0 > 0$ is an absolute constant:
\begin{align*}
    \chi_2(\mu_t \mmid \mu) \leq C_0 \exp\bigl(-\mf q(\gamma) \, t\bigr)\, \chi_2(\mu_0 \mmid \mu)\,,
\end{align*}
where the coefficient inside the exponent is 
\begin{align}\label{eq:poincare_decay}
    \mf q(\gamma) \deq \frac{C_{\msf{PI}}^{-1} \gamma}{C_0\, (C_{\msf{PI}}^{-1} +R^2 + \gamma^2)},
\end{align}
and the constant $R$ is 
\begin{align*}
    R = \begin{cases}
    0 & \text{if $U$ convex}\,, \\
    \sqrt{K} & \text{if}~\inf_{x \in \R^d} \nabla^2 U(x) \succeq -K I_d\,,\\
    \mf L \sqrt{d} & \text{if}~\norm{\nabla^2 U(x)}_{\rm op} \le \mf L\,(1+\norm{\nabla U(x)})~\text{for all}~x\in\R^d \,.
    \end{cases}
\end{align*}
\end{lemma}
\begin{remark*}
{
    In the strongly log-concave case, Lemma \ref{lem:cont_time_pi} actually yields a better decay of order $\sqrt{m}$ than Lemma \ref{lem:lyapunov_decay}, which has dependence $m/\sqrt{L}$. 
}
\end{remark*}

Our final result leverages the above accelerated convergence guarantees of ULD, and establishes the first bound for ULMC in R\'enyi divergence with an improved condition number dependence.

\begin{theorem}[{Convergence in $\eu{R}_q$ under \eqref{eq:pi}}]\label{thm:chip_pi}
    Let the potential be $(L,s)$-weakly smooth, satisfy \eqref{eq:pi} with constant $C_{\msf{PI}}$, and satisfy Assumption \ref{as:misc}.
    Let it also satisfy the additional technical condition Assumption \ref{as:l2_asmpt}. 
    Then, for $\xi \in (0, 1)$
    $$h= \tilde{\Theta} \Bigl(\frac{\gamma^{1/(2s)} \epsilon^{1/s} \xi^{1/s} \mf q(\gamma)^{1/(2s)}}{L^{1/s} d^{1/2} \,{(L \vee d)}^{1/(2s)}} \Bigr),$$ 
    the following holds for $\hat \mu_{Nh}$, the law of the $N$-th iterate of ULMC initialized at a centered Gaussian (variance specified in Appendix~\ref{scn:proofs}) for $q = 2 - \xi \in [1, 2)$ and 
    with $\mf q$ defined in \eqref{eq:poincare_decay}:
    \begin{align*}
        \eu R_q(\hat \mu_{Nh} \mmid \mu) \leq \epsilon^2 \qquad 
        \text{after} \qquad N 
        = \tilde{\Theta} \Bigl(\frac{L^{1/s}\, d^{1/2} \, {(L\vee d)}^{1+1/(2s)}}{\gamma^{1/(2s)}\, \epsilon^{1/s} \, \xi^{1/s} \,{\mf q(\gamma)}^{1+1/(2s)}} \Bigr)\quad \text{iterations}\,.
    \end{align*} 
\end{theorem}
\begin{remark*}
{
    The optimal choice is to take $\gamma \asymp \sqrt{C_{\msf{PI}}^{-1} + R^2}$.
    If the potential $U$ is convex, then we set $\gamma \asymp \mf q(1/\sqrt{C_{\msf{PI}}}) \asymp 1/\sqrt{C_{\msf{PI}}}$, which is known to be an optimal choice \cite{caoluwang2020underdamped}.
    As a result, in the convex and smooth case, the iteration complexity  has the condition number dependence $\kappa$, which improves upon the $\kappa^2$ dependence seen in \cite{chewi2021analysis}.
    The dependence on dimension $d$ and error tolerance $\epsilon$ are also improved. 
}
\end{remark*}

\section{Examples}\label{sec:applications}

\begin{example} {
    We consider the potential $U(x) = \sqrt{1+\norm{x}^2}$, which satisfies \eqref{eq:pi} with constant $\mc{O}(d)$~\cite{bobkov2003sphericallysymmetric} and is $(1, 1)$-smooth.
    Assuming the compact embedding condition of Assumption \ref{as:l2_asmpt}, Theorem \ref{thm:chip_pi} gives a complexity of $\widetilde{\mc{O}}(d^3 \xi^{-1}\epsilon^{-1})$ for $\epsilon^2$-guarantees in $\eu{R}_{2-\xi}$ after optimizing for $\gamma$, since in this case the potential is log-concave. In this case, the dimension dependence equates to that of the proximal sampler with rejection sampling \cite[Corollary 8]{chen2022improved}, which is $\widetilde{\mc{O}}(d^3)$; 
    it surpasses~\cite[Theorem 8]{chewi2021analysis}, which can only obtain $\widetilde{\mc{O}}(d^4\epsilon^{-2})$ for the same guarantees. However, it is important to note that the latter two works obtain these for any order of R\'enyi divergence and are not limited to order $q = 2- \xi < 2$, which cannot presently be obtained using our results for ULMC\@.}
\end{example}

\begin{example}\label{ex:slc}{
    Consider an $m$-strongly log-concave and $L$-log-smooth distribution. Non-trivial examples of this can be found in Bayesian regression (see e.g.,~\cite[Section 6]{dalalyan2017theoretical}); we will examine the first one, where $\pi(x) \propto \exp(-\norm{x-\bs a}^2/2) + \exp(-\norm{x+\bs a}^2/2)$ for some $\bs a \in \R^d: \norm{\bs a} =1/3$. 
    Here, our Theorem \ref{thm:kl_slc} gives a complexity of $N = \widetilde{\mc{O}}(d^{1/2}\epsilon^{-1})$ 
    to obtain a $\epsilon^2$-guarantee for the $\msf{KL}$ divergence. In contrast, the Hessian is $\nabla^2 U(x) = I_d - 4 \bs a \bs a^\top \exp(2x^\T \bs a)/(1+\exp(2x^\T\bs a))^2$, which has $L_{\msf H} \asymp d$, where $L_{\msf H}$ is the Lipschitz constant of the Hessian in the Frobenius norm. 
    Consequently,~\cite[Theorem 1]{maetal2021nesterovmcmc} is stated as $N = \widetilde{\mc{O}}(d^{1/2} L_H m^{-2} \epsilon^{-1})$, which in this case gives $N = \widetilde{\mc{O}}(d^{3/2} \epsilon^{-1})$ 
    to obtain the same $\epsilon^2$-accuracy guarantee.
    This is worse in the dimension-dependence. 
    Finally, it is possible to compare with the discretization bounds achieved in~\cite[Theorem 28]{ganeshtalwar2020renyi}, where in combination with our continuous time results (using the same proof technique as Theorem \ref{thm:kl_slc}) to yield 
    $N = \widetilde{O}(d^{1/2}\epsilon^{-2})$ iterations, which is suboptimal in the order of $\epsilon$, but has the same dimension dependence.}
\end{example}

\begin{example}{
    We can analyze $L$-smooth distributions satisfying a log-Sobolev inequality with parameter $C_{\msf{LSI}}$. 
    One such instance arises when considering any bounded perturbation of a strongly convex potential. 
    In this case, let $U_{ \bs a}$ be the potential of the target in Example \ref{ex:slc}. Then consider a target with modified potential $U_{\bs a} + f$, with $\sup_{x}\abs{f(x)} \vee \norm{\nabla f(x)} \vee \norm{\nabla^2 f(x)}_{\msf{op}} \leq \mf B$ for some $\mf B < \infty$, and let $\nabla^2 f$ be $\mc{O}(d)$-Frobenius Lipschitz. We can bound the log-Sobolev constant of this potential using the Holley--Stroock Lemma \cite{holley1986logarithmic}. Let this new potential have condition number $\kappa$. We achieve $\epsilon$-accuracy in $\msf{TV}$ distance with  $N = \widetilde{\mc{O}}(\kappa^{3/2} d^{1/2} \epsilon^{-1})$. For comparison, the previous bound  \cite[Theorem 1]{maetal2021nesterovmcmc} gives $N = \tilde{\mc{O}}(\kappa^2 d^{3/2}\epsilon^{-1})$ to arrive at the same guarantee in $\msf{TV}$, which is worse in the dimension.
    However, note that the guarantees in \cite[Theorem 1]{maetal2021nesterovmcmc} are in $\msf{KL}$, which is stronger than $\msf{TV}$. 
    Finally, we note that \cite{ganeshtalwar2020renyi} requires strong log-concavity, and hence cannot provide a guarantee in this setting.}
\end{example}

\begin{example}
{
    Consider a $(1,s)$-weakly log-smooth target that is log-concave and satisfies a Poincar\'e inequality with $C_{\msf{PI}} = \mc{O}(d).$ Consequently, Theorem \ref{thm:chip_pi} yields $N = \tilde{\mc{O}}(d^{2+1/s} \xi^{-1/s} \epsilon^{-1/s})$ to obtain $\epsilon^2$-guarantees for $\eu R_{2-\xi}$. \cite[Theorem 7]{chewi2021analysis} yields $N=\tilde{\mc{O}}(d^{3+2/s} \epsilon^{-2/s})$ for the same guarantees, which is worse in both parameters. On the other hand, take the specific case of a distribution with potential $U(x) = \norm{x}^\alpha$, which has $C_{\msf{PI}} = \mc{O}(d^{2/\alpha - 1}) $ \cite{bobkov2003sphericallysymmetric}, is log-convex and $(1, \alpha-1)$-weakly log-smooth. Consequently, Theorem \ref{thm:chip_pi} yields  
    $N = \tilde{\mc{O}}( d^{\alpha / (\alpha-1)} \xi^{-1/(\alpha-1)} \epsilon^{-1/(\alpha-1)} )$
    for $\epsilon^2$-accuracy guarantees in $\eu R_{2-\xi}$ divergence. 
    This is worse by a factor of $d$ than the rate estimate obtained in \cite[Example 9]{chewi2021analysis}, as they leverage a stronger class of functional inequalities that interpolate between \eqref{eq:pi} and \eqref{eq:LSI}, whereas our analysis cannot capture this improvement. Our convergence guarantee is still better in terms of $\epsilon$-dependence.} 
\end{example}

\section{Proof Sketches}\label{sec:proof_sketch}

\subsection{Continuous Time Results}

For results under both the Poincar\'e and log-Sobolev inequalities, we leverage the existing results as stated in \cite{caoluwang2020underdamped, maetal2021nesterovmcmc}, which we present in Lemmas \ref{lem:lyapunov_decay} and \ref{lem:cont_time_pi}. These allow us to bound $\chi_2(\mu_t \mmid \mu)$, $\msf{KL}(\mu_t \mmid \mu)$ with exponentially decaying quantities.

With the additional assumption of strong convexity, we can obtain a contraction in an alternate system of coordinates $(\phi, \psi) \defeq \mc M(x,v) \defeq (x, x+\frac{2}{\gamma}\, v)$ (see Appendix \ref{sec:continuous_time}). This allows us to consider the distributions of the continuous time iterates and the target in these alternate coordinates $\mu_t^{\mc M}, \mu^{\mc M}$ respectively. From this, we obtain the following proposition. 

\begin{proposition}[{Log-Sobolev Inequality Along the Trajectory}]\label{lem:lsi_ct}
    Suppose $U$ is $m$-strongly convex and $L$-smooth. Let $\mu_t^{\mc M}$ now denote the law of the continuous-time underdamped Langevin diffusion with $\gamma = c\sqrt{L}$ for $c \geq \sqrt{2}$ in the $(\phi,\psi)$ coordinates. Suppose the initial distribution $\mu_0$ has \eqref{eq:LSI} constant (in the altered coordinates) $C_{\msf{LSI}}(\mu_0^{\mc M})$, then $\{\mu_t^{\mc M}\}_{t \geq 0}$ satisfies \eqref{eq:LSI} with constant that can be uniformly upper bounded by        
    \begin{align*}
        C_{\msf{LSI}}(\mu_t^{\mc M})
        &\le \exp\Bigl( - m\sqrt{\frac{2}{L}}\,t\Bigr)\, C_{\msf{LSI}}(\mu_0^{\mc M}) + \frac{2}{m}\,.
    \end{align*}
\end{proposition}
The main idea behind the proof of this proposition is to analyze the discretization \eqref{eq:ulmc_sde} of the underdamped Langevin diffusion in the coordinates $(\phi, \psi)$. Note that this can be written in the following form, for some matrix $\overline{\Sigma} \in \R^{2d \times 2d}$ and function $\bar F: \R^{d} \times \R^{d} \to \R^{d} \times \R^{d}$,
\begin{align*}
    (\phi_{(k+1)h}, \psi_{(k+1)h})
    &\eqdist \bar F(\phi_{kh}, \psi_{kh}) + \mc N(0, \overline{\Sigma})\,.
\end{align*} 
This is the composition of a deterministic function $\bar F$ giving the mean of the next iterate of ULMC started at $(\phi,\psi)$, followed by addition with a Gaussian distribution giving the variance of the resulting iterate. In particular, we show that for coordinates $(\phi(x, v), \psi(x, v)) \defeq (x, x + \frac{2}{\gamma} v)$, we can find an almost sure strict contraction under $\bar F$ in the sense that
\begin{align*}
    \norm{\bar F}_{\text{Lip}} \leq 1 - \frac{m}{\sqrt{2L}}\, h + \mc{O}(Lh^2)\,,
\end{align*}
where by abuse of notation $\bar F: \R^{2d} \to \R^{2d}$, and the seminorm $\norm{g}_{\text{Lip}}$ of a function $g: \R^{2d} \to \R^{2d}$ refers to the Lipschitz constant of the function.

Since $\bar F$ is a contraction for small enough $h$, each push forward improves the log-Sobolev constant by a multiplicative factor \cite[Lemma 19]{vempalawibisono2019ula}. 
At the same time, a Gaussian convolution can only worsen the log-Sobolev constant by an additive constant \cite[Corollary 3.1]{chafai2004phientropies}.
Subsequently, the log-Sobolev constant at each iterate forms a (truncated) geometric sum, and therefore can be bounded by the infinite series. This incidentally can be used to bound the log-Sobolev constant of the ULMC iterates. 
Taking an appropriate limit of $h \to 0$ while keeping $Nh = t$, we arrive at the stated bound in the proposition. Consequently, considering the decomposition of the $\msf{KL}$, a simple application of Cauchy--Schwarz tells us that
\begin{align*}
        \msf{KL}(\hat \mu_t^{\mc M} \mmid \mu^{\mc M}) &= \int \log \frac{\hat \mu_t^{\mc M}}{\mu^{\mc M}}\, \D \hat \mu_t^{\mc M}= \msf{KL}(\hat \mu_t^{\mc M} \mmid \mu_t^{\mc M}) + \int \log \frac{\mu_t^{\mc M}}{\mu^{\mc M}} \, \D \hat \mu_t^{\mc M} \\
        &\leq \msf{KL}(\hat \mu_t^{\mc M} \mmid \mu_t^{\mc M}) 
 + \msf{KL}(\mu_t^{\mc M} \mmid \mu^{\mc M}) + \sqrt{\chi^2(\hat \mu_t^{\mc M} \mmid \mu_t^{\mc M})\times \msf{var}_{\mu_t^{\mc M}}\Bigl(\log \frac{\mu_t^{\mc M}}{\mu^{\mc M}}\Bigr)}\,.
    \end{align*}
    
The log-Sobolev inequality for $\mu_t^{\mc M}$ implies a Poincar\'e inequality, which allows us to bound the variance term by the Fisher information $\msf{FI}(\mu_t^{\mc M} \mmid \mu^{\mc M}) = \E_{\mu_t^{\mc M}}\norm{\nabla \log (\mu_t^{\mc M} /\mu^{\mc M})}^2$. This can be bounded by the same entropic hypocoercivity argument from \cite{maetal2021nesterovmcmc} that is used to generate our $\msf{TV}$ bounds, while the remaining two terms are handled respectively via the discretization analysis and again the entropic hypocoercivity argument.

\subsection{Discretization Analysis}

The main result we use to control the discretization error can be found below. 

\begin{proposition}\label{prop:main_disc_bd}
    Let $(\hat\mu_t)_{t\ge 0}$ denote the law of \eqref{eq:ulmc_sde} and let ${(\mu_t)}_{t\ge 0}$ denote the law of the continuous-time underdamped Langevin diffusion~\eqref{eq:ULD}, both initialized at some $\mu_0$. Assume that the potential $U$ is $(L,s)$-weakly smooth.
    If the step size $h$ satisfies
    \begin{align}\label{eq:h_requirement}
        h =\widetilde{\mc O}_s \Bigl(\frac{\gamma^{1/(2s)}\, \epsilon^{1/s}}{L^{1/s}\, T^{1/(2s)} \, (d +\eu R_2(\mu_0 \mmid \mu^{(a)}))^{1/2}}\Bigr)\,,
    \end{align}
    where the notation $\widetilde{\mc O}_s$ hides constants depending on $s$ as well as polylogarithmic factors including $\log N$, and $\mu^{(a)}$ is a modified target distribution (see Appendix~\ref{scn:subgaussianty} for details), then
    \begin{align*}
        \eu R_q(\hat \mu_T \mmid \mu_T)
        &\le \epsilon^2\,.
    \end{align*}
\end{proposition}

\begin{remark*}
{The condition on $h$ is dependent on $N$ only through logarithmic factors. Secondly, this is shown under generic assumptions, and can be combined with continuous-time results in $\eu R_q$ in any setting, such as the log-Sobolev or Lata\l a--Oleszkiewicz inequalities seen in \cite{chewi2021analysis}. 
}
\end{remark*}

We outline the proof of this result below. Similar to the work of \cite{chewi2021analysis}, we first invoke the data processing inequality, allowing us to bound the R\'enyi between the time marginal distributions of the iterates with R\'enyi between the path measures 
\begin{align*}
    \eu R_q(\hat \mu_T \mmid \mu_T) \leq \eu R_q(P_T \mmid Q_T)\,,
\end{align*}
where $P_T, Q_T$ are probability measures of \eqref{eq:ulmc_sde}, \eqref{eq:ULD} respectively on the space of paths $C([0,T], \R^{2d})$. 
Subsequently, we invoke Girsanov's theorem, which allows us to exactly bound the pathwise divergence by the difference between the drifts of the two processes:
\begin{align*}
    \eu R_{2q}(P_T \mmid Q_T) \lesssim \log \E \exp \Bigl(\frac{4q^2}{\gamma}\int_0^T \norm{\nabla U(x_t)  -\nabla U (x_{\floor{t/h} h})}^2 \, \D t \Bigr)\,.
\end{align*}
It remains to bound the term inside the expectation. We achieve this by conditioning on the event that $\sup_{t \in [0, T]} \norm{x_t - x_{\floor{t/h} h}}^2$ is bounded by a vanishing quantity as $h \to 0$, which we must demonstrate occurs with sufficiently high probability.
To show this, we begin with a single-step analysis, i.e., we bound the above for $T \leq h$. Compared to LMC, the main gain in this analysis is that the SDEs \eqref{eq:ULD} and \eqref{eq:ulmc_sde} match exactly in the position coordinate, while the difference between the drifts manifests solely in the momentum. After integration of the momentum, the order of error is better in the position coordinate (the dominant term is $\mc{O}(dh^2)$ compared to $\mc{O}(dh)$ seen in~\cite[Lemma 24]{chewi2021analysis}).

The technique for extending this analysis from a single step to the full time interval follows closely that seen in \cite{chewi2021analysis}. In particular, we obtain a dependence for $\norm{x_{t}}$ on $\norm{x_{kh}}$ in the interval $t \in [kh, (k + 1) h)$. Controlling the latter is quite complicated when the potential satisfies only a Poincar\'e inequality, since it is equivalent to showing sub-Gaussian tail bounds on the iterates, while the target itself is not sub-Gaussian in the position coordinate. 
By comparing against an auxiliary potential, we can show that for our choice of initialization, the iterates remain sub-Gaussian for all iterations up to $N$ (albeit with a growing constant). 
Finally, this allows us to recover our discretization result in the proposition above.

\section{Conclusion}\label{sec:conclusion}

This work provides state-of-the-art convergence guarantees for underdamped Langevin Monte Carlo algorithm in several regimes.
Our discretization analysis (Proposition \ref{prop:main_disc_bd}) in particular is generic and can be extended to any order of R\'enyi, under various conditions on the potential (Lata\l a--Oleszkiewicz, weak smoothness, etc.). 
Consequently, our results serve as a key step towards a complete understanding of the ULMC algorithm. 
However, limitations of the current continuous-time techniques do not permit us to obtain stronger iteration complexity results. 
More specifically, it is not understood how to analyze R\'enyi divergence of order greater than $2$, or if hypercontractive decay is possible when the potential satisfies a log-Sobolev inequality. 
Secondly, our discretization approach via Girsanov is currently suboptimal in the condition number (a fact noted in \cite{chewi2021analysis}), and thus does not obtain the expected dependence of $\sqrt{\kappa}$ after discretization. 
An improvement in the proof techniques would be necessary to sharpen this result. 
We believe the results and techniques developed in this work will be of interest to stimulate future research.

\section*{Acknowledgements}

We thank Jason M.\ Altschuler, Alain Durmus, and Aram-Alexandre Pooladian for helpful conversations. 
KB was supported by NSF grant DMS-2053918.
SC was supported by the NSF TRIPODS program (award DMS-2022448).
MAE was supported by NSERC Grant [2019-06167], the Connaught New Researcher Award, the CIFAR AI Chairs program, and the CIFAR AI Catalyst grant.
ML was supported by the Ontario Graduate Scholarship and Vector Institute. 

\newpage

\printbibliography

@article{vishnoi2021introduction,
  title={{An introduction to Hamiltonian Monte Carlo method for sampling}},
  author={Vishnoi, Nisheeth K},
  journal={arXiv preprint arXiv:2108.12107},
  year={2021}
}

@article{foster2022high,
  title={{High order splitting methods for SDEs satisfying a commutativity condition}},
  author={Foster, James and Reis, Goncalo dos and Strange, Calum},
  journal={arXiv preprint arXiv:2210.17543},
  year={2022}
}

@article{bernard2022hypocoercivity,
  title={Hypocoercivity with Schur complements},
  author={Bernard, Etienne and Fathi, Max and Levitt, Antoine and Stoltz, Gabriel},
  journal={Annales Henri Lebesgue},
  volume={5},
  pages={523--557},
  year={2022}
}

@article{johnston2023kinetic,
  title={{Kinetic Langevin MCMC Sampling Without Gradient Lipschitz Continuity--the Strongly Convex Case}},
  author={Johnston, Tim and Lytras, Iosif and Sabanis, Sotirios},
  journal={arXiv preprint arXiv:2301.08039},
  year={2023}
}

@article{cheng2018sharp,
  title={Sharp convergence rates for Langevin dynamics in the nonconvex setting},
  author={Cheng, Xiang and Chatterji, Niladri S and Abbasi-Yadkori, Yasin and Bartlett, Peter L and Jordan, Michael I},
  journal={arXiv preprint arXiv:1805.01648},
  year={2018}
}

@article{jordan1998variational,
  title={The variational formulation of the {F}okker--{P}lanck equation},
  author={Jordan, Richard and Kinderlehrer, David and Otto, Felix},
  journal={SIAM Journal on Mathematical Analysis},
  volume={29},
  number={1},
  pages={1--17},
  year={1998},
  publisher={SIAM}
}

@inproceedings{nesterov1983method,
  title={A method of solving a convex programming problem with convergence rate ${O}(\frac{1}{k^2})$},
  author={Nesterov, Yurii E.},
  booktitle={Doklady Akademii Nauk},
  volume={269},
  number={3},
  pages={543--547},
  year={1983},
  organization={Russian Academy of Sciences}
}

@book{oksendal2013stochastic,
  title={Stochastic differential equations: an introduction with applications},
  author={Oksendal, Bernt},
  year={2013},
  publisher={Springer Science \& Business Media}
}

@article{maetal2021nesterovmcmc,
author = {Yi-An Ma and Niladri S. Chatterji and Xiang Cheng and Nicolas Flammarion and Peter L. Bartlett and Michael I. Jordan},
title = {Is there an analog of {N}esterov acceleration for gradient-based {MCMC}?},
volume = {27},
journal = {Bernoulli},
number = {3},
publisher = {Bernoulli Society for Mathematical Statistics and Probability},
pages = {1942 -- 1992},
year = {2021},
}

@article{caoluwang2020underdamped,
      title={On explicit {$L^2$}-convergence rate estimate for underdamped {L}angevin dynamics}, 
      author={Yu Cao and Jianfeng Lu and Lihan Wang},
      year={2020},
      eid={arXiv:1908.04746},
      journal={arXiv e-prints},
}

@inproceedings{vempalawibisono2019ula,
 author = {Vempala, Santosh and Wibisono, Andre},
 booktitle = {Advances in Neural Information Processing Systems},
 editor = {H. Wallach and H. Larochelle and A. Beygelzimer and F. d\textquotesingle Alch\'{e}-Buc and E. Fox and R. Garnett},
 pages = {},
 publisher = {Curran Associates, Inc.},
 title = {Rapid convergence of the unadjusted {L}angevin algorithm: isoperimetry suffices},
 volume = {32},
 year = {2019}
}

@article{herau2006hypocoercivity,
  title={Hypocoercivity and exponential time decay for the linear inhomogeneous relaxation {B}oltzmann equation},
  author={H{\'e}rau, Fr{\'e}d{\'e}ric},
  journal={Asymptotic Analysis},
  volume={46},
  number={3-4},
  pages={349--359},
  year={2006},
  publisher={IOS Press}
}

@inproceedings{ganeshtalwar2020renyi,
 author = {Ganesh, Arun and Talwar, Kunal},
 booktitle = {Advances in Neural Information Processing Systems},
 editor = {H. Larochelle and M. Ranzato and R. Hadsell and M. F. Balcan and H. Lin},
 pages = {7222--7233},
 publisher = {Curran Associates, Inc.},
 title = {Faster differentially private samplers via {R}\'{e}nyi divergence analysis of discretized {L}angevin {MCMC}},
 volume = {33},
 year = {2020}
}

@InProceedings{erdhoszha22chisq,
  title = 	 {Convergence of {L}angevin {M}onte {C}arlo in chi-squared and {R}\'enyi divergence},
  author =       {Erdogdu, Murat A. and Hosseinzadeh, Rasa and Zhang, Shunshi},
  booktitle = 	 {Proceedings of The 25th International Conference on Artificial Intelligence and Statistics},
  pages = 	 {8151--8175},
  year = 	 {2022},
  editor = 	 {Camps-Valls, Gustau and Ruiz, Francisco J. R. and Valera, Isabel},
  volume = 	 {151},
  series = 	 {Proceedings of Machine Learning Research},
  month = 	 {3},
  publisher =    {PMLR},
}

@article {mouetal22langevin,
    AUTHOR = {Mou, Wenlong and Flammarion, Nicolas and Wainwright, Martin J.
              and Bartlett, Peter L.},
     TITLE = {Improved bounds for discretization of {L}angevin diffusions:
              near-optimal rates without convexity},
   JOURNAL = {Bernoulli},
  FJOURNAL = {Bernoulli. Official Journal of the Bernoulli Society for
              Mathematical Statistics and Probability},
    VOLUME = {28},
      YEAR = {2022},
    NUMBER = {3},
     PAGES = {1577--1601},
}

@InProceedings{erdogduhosseinzadeh2021tailgrowth,
  title = 	 {On the convergence of {L}angevin {M}onte {C}arlo: the interplay between tail growth and smoothness},
  author =       {Erdogdu, Murat A. and Hosseinzadeh, Rasa},
  booktitle = 	 {Proceedings of Thirty Fourth Conference on Learning Theory},
  pages = 	 {1776--1822},
  year = 	 {2021},
  editor = 	 {Belkin, Mikhail and Kpotufe, Samory},
  volume = 	 {134},
  series = 	 {Proceedings of Machine Learning Research},
  publisher =    {PMLR},
}

@book {bakrygentilledoux2014,
    AUTHOR = {Bakry, Dominique and Gentil, Ivan and Ledoux, Michel},
     TITLE = {Analysis and geometry of {M}arkov diffusion operators},
    SERIES = {Grundlehren der Mathematischen Wissenschaften [Fundamental
              Principles of Mathematical Sciences]},
    VOLUME = {348},
 PUBLISHER = {Springer, Cham},
      YEAR = {2014},
     PAGES = {xx+552},
}

@inproceedings{mironov2017renyi,
  title={{R}{\'e}nyi differential privacy},
  author={Mironov, Ilya},
  booktitle={2017 {I}{E}{E}{E} 30th {C}omputer {S}ecurity {F}oundations {S}ymposium ({C}{S}{F})},
  pages={263--275},
  year={2017},
  organization={IEEE}
}

@article{kolmogoroff1934zufallige,
  title={Zufallige bewegungen (zur theorie der {B}rownschen bewegung)},
  author={Kolmogorov, Andrey},
  journal={Annals of Mathematics},
  pages={116--117},
  year={1934},
  publisher={JSTOR}
}

@article{dalalyan2020sampling,
  title={On sampling from a log-concave density using kinetic {L}angevin diffusions},
  author={Dalalyan, Arnak S. and Riou-Durand, Lionel},
  journal={Bernoulli},
  volume={26},
  number={3},
  pages={1956--1988},
  year={2020},
  publisher={Bernoulli Society for Mathematical Statistics and Probability}
}

@article{foster2021shifted,
  title={The shifted {ODE} method for underdamped {L}angevin {MCMC}},
  author={Foster, James and Lyons, Terry and Oberhauser, Harald},
  journal={arXiv preprint arXiv:2101.03446},
  year={2021}
}

@article{monmarche2021high,
  title={High-dimensional {MCMC} with a standard splitting scheme for the underdamped {L}angevin diffusion.},
  author={Monmarch{\'e}, Pierre},
  journal={Electronic Journal of Statistics},
  volume={15},
  number={2},
  pages={4117--4166},
  year={2021},
  publisher={Institute of Mathematical Statistics and Bernoulli Society}
}

@article{albritton2019variational,
  title={Variational methods for the kinetic {F}okker--{P}lanck equation},
  author={Albritton, Dallas and Armstrong, Scott and Mourrat, Jean-Christophe and Novack, Matthew},
  journal={arXiv preprint arXiv:1902.04037},
  year={2019}
}

@article{dolbeault2009hypocoercivity,
  title={Hypocoercivity for kinetic equations with linear relaxation terms},
  author={Dolbeault, Jean and Mouhot, Cl{\'e}ment and Schmeiser, Christian},
  journal={Comptes Rendus Mathematique},
  volume={347},
  number={9-10},
  pages={511--516},
  year={2009},
  publisher={Elsevier}
}

@article{dolbeault2015hypocoercivity,
  title={Hypocoercivity for linear kinetic equations conserving mass},
  author={Dolbeault, Jean and Mouhot, Cl{\'e}ment and Schmeiser, Christian},
  journal={Transactions of the American Mathematical Society},
  volume={367},
  number={6},
  pages={3807--3828},
  year={2015}
}

@InProceedings{chen2022improved,
  title = 	 {Improved analysis for a proximal algorithm for sampling},
  author =       {Chen, Yongxin and Chewi, Sinho and Salim, Adil and Wibisono, Andre},
  booktitle = 	 {Proceedings of Thirty Fifth Conference on Learning Theory},
  pages = 	 {2984--3014},
  year = 	 {2022},
  editor = 	 {Loh, Po-Ling and Raginsky, Maxim},
  volume = 	 {178},
  series = 	 {Proceedings of Machine Learning Research},
  month = 	 {7},
  publisher =    {PMLR},
}

@article{kobyzev2020normalizing,
  title={Normalizing flows: an introduction and review of current methods},
  author={Kobyzev, Ivan and Prince, Simon JD and Brubaker, Marcus A.},
  journal={IEEE Transactions on Pattern Analysis and Machine Intelligence},
  volume={43},
  number={11},
  pages={3964--3979},
  year={2020},
  publisher={IEEE}
}

@article{von2011bayesian,
  title={Bayesian inference in physics},
  author={Von Toussaint, Udo},
  journal={Reviews of Modern Physics},
  volume={83},
  number={3},
  pages={943},
  year={2011},
  publisher={APS}
}

@incollection{johannes2010mcmc,
  title={{MCMC} methods for continuous-time financial econometrics},
  author={Johannes, Michael and Polson, Nicholas},
  booktitle={Handbook of financial econometrics: applications},
  pages={1--72},
  year={2010},
  publisher={Elsevier}
}

@article{hooton1981compact,
  title={Compact {S}obolev imbeddings on finite measure spaces},
  author={Hooton, James G.},
  journal={Journal of Mathematical Analysis and Applications},
  volume={83},
  number={2},
  pages={570--581},
  year={1981},
  publisher={Elsevier}
}

@article{chandra1976linear,
  title={Linear generalizations of {G}ronwall’s inequality},
  author={Chandra, Jagdish and Davis, Paul W.},
  journal={Proceedings of the American Mathematical Society},
  volume={60},
  number={1},
  pages={157--160},
  year={1976}
}

@article{dalalyan2017theoretical,
  title={Theoretical guarantees for approximate sampling from smooth and log-concave densities},
  author={Dalalyan, Arnak S.},
  journal={Journal of the Royal Statistical Society: Series B (Statistical Methodology)},
  volume={79},
  number={3},
  pages={651--676},
  year={2017},
  publisher={Wiley Online Library}
}

@article{holley1986logarithmic,
    AUTHOR = {Holley, Richard and Stroock, Daniel},
     TITLE = {Logarithmic {S}obolev inequalities and stochastic {I}sing
              models},
   JOURNAL = {J. Statist. Phys.},
  FJOURNAL = {Journal of Statistical Physics},
    VOLUME = {46},
      YEAR = {1987},
    NUMBER = {5-6},
     PAGES = {1159--1194},
}

@article{apers2022hamiltonian,
  title={Hamiltonian {M}onte {C}arlo for efficient {G}aussian sampling: long and random steps},
  author={Apers, Simon and Gribling, Sander and Szil{\'a}gyi, D{\'a}niel},
  journal={arXiv preprint arXiv:2209.12771},
  year={2022}
}

@article{wang2022accelerating,
  title={Accelerating {H}amiltonian {M}onte {C}arlo via {C}hebyshev integration time},
  author={Wang, Jun-Kun and Wibisono, Andre},
  journal={arXiv preprint arXiv:2207.02189},
  year={2022}
}

@article{bou2022unadjusted,
  title={Unadjusted {H}amiltonian {MCMC} with stratified {M}onte {C}arlo time integration},
  author={Bou-Rabee, Nawaf and Marsden, Milo},
  journal={arXiv preprint arXiv:2211.11003},
  year={2022}
}

@article{hormander1967hypoelliptic,
  title={Hypoelliptic second order differential equations},
  author={H{\"o}rmander, Lars},
  journal={Acta Mathematica},
  volume={119},
  pages={147--171},
  year={1967},
  publisher={Institut Mittag-Leffler}
}

@article {villani2009hypocoercivity,
    AUTHOR = {Villani, C\'{e}dric},
     TITLE = {Hypocoercivity},
   JOURNAL = {Mem. Amer. Math. Soc.},
  FJOURNAL = {Memoirs of the American Mathematical Society},
    VOLUME = {202},
      YEAR = {2009},
    NUMBER = {950},
     PAGES = {iv+141},
}

@article{villani2002limites,
  title={Limites hydrodynamiques de l’{\'e}quation de {B}oltzmann},
  author={Villani, C{\'e}dric},
  journal={Ast{\'e}risque, SMF},
  volume={282},
  pages={365--405},
  year={2002}
}

@InProceedings{dalalyan2017analogy,
  title = 	 {Further and stronger analogy between sampling and optimization: {L}angevin {M}onte {C}arlo and gradient descent},
  author = 	 {Dalalyan, Arnak S.},
  booktitle = 	 {Proceedings of the 2017 Conference on Learning Theory},
  pages = 	 {678--689},
  year = 	 {2017},
  editor = 	 {Kale, Satyen and Shamir, Ohad},
  volume = 	 {65},
  series = 	 {Proceedings of Machine Learning Research},
  publisher =    {PMLR},
}

@inproceedings{shenlee2019randomizedmidpoint,
 author = {Shen, Ruoqi and Lee, Yin Tat},
 booktitle = {Advances in Neural Information Processing Systems},
 editor = {H. Wallach and H. Larochelle and A. Beygelzimer and F. d\textquotesingle Alch\'{e}-Buc and E. Fox and R. Garnett},
 pages = {},
 publisher = {Curran Associates, Inc.},
 title = {The randomized midpoint method for log-concave sampling},
 volume = {32},
 year = {2019}
}

@inproceedings{hebalasubramanianerdogdu2020rm,
 author = {He, Ye and Balasubramanian, Krishnakumar and Erdogdu, Murat A.},
 booktitle = {Advances in Neural Information Processing Systems},
 editor = {H. Larochelle and M. Ranzato and R. Hadsell and M. F. Balcan and H. Lin},
 pages = {7366--7376},
 publisher = {Curran Associates, Inc.},
 title = {On the ergodicity, bias and asymptotic normality of randomized midpoint sampling method},
 volume = {33},
 year = {2020}
}

@article {chen2021kls,
    AUTHOR = {Chen, Yuansi},
     TITLE = {An almost constant lower bound of the isoperimetric
              coefficient in the {KLS} conjecture},
   JOURNAL = {Geom. Funct. Anal.},
  FJOURNAL = {Geometric and Functional Analysis},
    VOLUME = {31},
      YEAR = {2021},
    NUMBER = {1},
     PAGES = {34--61},
}

@article{chenchewinilesweed2021dimfreelsi,
title = {Dimension-free log-{S}obolev inequalities for mixture distributions},
journal = {Journal of Functional Analysis},
volume = {281},
number = {11},
pages = {109236},
year = {2021},
author = {Hong-Bin Chen and Sinho Chewi and Jonathan Niles-Weed},
}

@InProceedings{raginskyrakhlintelgarsky2017sgld,
  title = 	 {Non-convex learning via stochastic gradient {L}angevin dynamics: a nonasymptotic analysis},
  author = 	 {Raginsky, Maxim and Rakhlin, Alexander and Telgarsky, Matus},
  booktitle = 	 {Proceedings of the 2017 Conference on Learning Theory},
  pages = 	 {1674--1703},
  year = 	 {2017},
  editor = 	 {Kale, Satyen and Shamir, Ohad},
  volume = 	 {65},
  series = 	 {Proceedings of Machine Learning Research},
  publisher =    {PMLR},
}

@book {boucheronlugosimassart2013concentration,
    AUTHOR = {Boucheron, St\'{e}phane and Lugosi, G\'{a}bor and Massart, Pascal},
     TITLE = {Concentration inequalities},
      NOTE = {A nonasymptotic theory of independence,
              With a foreword by Michel Ledoux},
 PUBLISHER = {Oxford University Press, Oxford},
      YEAR = {2013},
     PAGES = {x+481},
}

@article {chafai2004phientropies,
    AUTHOR = {Chafai, Djalil},
     TITLE = {Entropies, convexity, and functional inequalities: on
              {$\Phi$}-entropies and {$\Phi$}-{S}obolev inequalities},
   JOURNAL = {J. Math. Kyoto Univ.},
  FJOURNAL = {Journal of Mathematics of Kyoto University},
    VOLUME = {44},
      YEAR = {2004},
    NUMBER = {2},
     PAGES = {325--363},
}

@incollection {bobkov2003sphericallysymmetric,
    AUTHOR = {Bobkov, Sergey G.},
     TITLE = {Spectral gap and concentration for some spherically symmetric
              probability measures},
 BOOKTITLE = {Geometric aspects of functional analysis},
    SERIES = {Lecture Notes in Math.},
    VOLUME = {1807},
     PAGES = {37--43},
 PUBLISHER = {Springer, Berlin},
      YEAR = {2003},
}

@inproceedings{erdogdu2018global,
  title={Global non-convex optimization with discretized diffusions},
  author={Erdogdu, Murat A and Mackey, Lester and Shamir, Ohad},
  booktitle={Proceedings of the 32nd International Conference on Neural Information Processing Systems},
  pages={9694--9703},
  year={2018}
}

@article{chewi2021analysis,
  title={Analysis of {L}angevin {M}onte {C}arlo from {P}oincar\'e to log-{S}obolev},
  author={Chewi, Sinho and Erdogdu, Murat A. and Li, Mufan Bill and Shen, Ruoqi and Zhang, Matthew},
  journal={arXiv preprint arXiv:2112.12662},
  year={2021}
}

@article{roussel2018spectral,
  title={Spectral methods for {L}angevin dynamics and associated error estimates},
  author={Roussel, Julien and Stoltz, Gabriel},
  journal={ESAIM: Mathematical Modelling and Numerical Analysis},
  volume={52},
  number={3},
  pages={1051--1083},
  year={2018},
  publisher={EDP Sciences}
}

@inproceedings{cheng2018underdamped,
  title={Underdamped {L}angevin {MCMC}: a non-asymptotic analysis},
  author={Cheng, Xiang and Chatterji, Niladri S. and Bartlett, Peter L. and Jordan, Michael I.},
  booktitle={Conference on Learning Theory},
  pages={300--323},
  year={2018},
  organization={PMLR}
}

@article{li2019stochastic,
  title={Stochastic runge-kutta accelerates langevin monte carlo and beyond},
  author={Li, Xuechen and Wu, Yi and Mackey, Lester and Erdogdu, Murat A},
  journal={Advances in neural information processing systems},
  volume={32},
  year={2019}
}

@book {chewisamplingbook,
    AUTHOR = {Sinho Chewi},
     TITLE = {Log-concave sampling},
      NOTE = {Book draft available at \url{https://chewisinho.github.io/}},
      YEAR = {2023},
}

\newpage 
\appendix

\section{Explicit Form for the Underdamped Langevin Diffusion}\label{sec:explicit_form}
Recall that we evolve $(x_t, v_t)$ for time $t \in [kh, (k+1)h)$ explicitly according to the SDE \eqref{eq:ulmc_sde}, which we repeat here for convenience:
\begin{align}
    \D x_{t} &\defeq v_t \, \D t\,, \\
    \D v_{t} &\defeq - \gamma v_t + \nabla U(x_{kh}) \, \D t + \sqrt{2\gamma} \, \D B_t\,.
\end{align}
Consequently, since we fix the position $x_{kh}$ in the non-linear term, this permits an explicit solution
\begin{align}
    x_{(k+1)h}
    &= x_{kh} + \gamma^{-1}\, ( 1 - \exp(-\gamma h) )\, v_{kh} - \gamma^{-1}\, (h - \gamma^{-1}\,
    (1-\exp(-\gamma h))) \,\nabla U(x_{kh}) + W_{k}^x\,, \\
    v_{(k+1)h}
    &=\exp(-\gamma h)\, v_{kh} - \gamma^{-1}\, (1 - \exp(-\gamma h))\, \nabla U(x_{kh}) + W_{k}^v\,,
\end{align}
where $(W_k^x, W_k^v)_{k\in\N}$ is an independent sequence of pairs of variables, where each pair has the joint distribution
\begin{align*}
    \begin{bmatrix}
        W_k^x \\
        W_k^v
    \end{bmatrix} &\sim \mc{N} \biggl(0, \begin{bmatrix}
        \frac{2}{\gamma}\, (h - \frac{2}{\gamma}\, (1-\exp(-\gamma h)) + \frac{1}{2\gamma}\, (1-\exp(-2\gamma h))) & *  \\
        \frac{1}{\gamma} \,(1 - 2\exp(-\gamma h) + \exp(-2\gamma h)) & 1-\exp(-2\gamma h) 
    \end{bmatrix}\biggr)\,,
\end{align*}
where $*$ is identical to the bottom left entry.

\section{Continuous-Time Results} \label{sec:continuous_time}

\subsection{Entropic Hypocoercivity}\label{scn:entropic_hypo}

Our proof of Lemma~\ref{lem:lyapunov_decay} is based on adapting the argument on the decay of a Lyapunov function from \cite{maetal2021nesterovmcmc} (based on entropic hypocoercivity, see~\cite{villani2009hypocoercivity}) and combining it with a time change argument~\cite[Lemma 1]{dalalyan2020sampling}. We provide the details below for completeness.

\begin{proof}{\textbf{of Lemma~}\ref{lem:lyapunov_decay}} First note that variables $x_t, v_t$ with $\gamma = 2\sqrt{2L}$ following \eqref{eq:ulmc_sde} can be changed into $(\tilde x_t, \tilde v_t) = (x_{t\sqrt{\xi}}, \frac{1}{\sqrt{\xi}}\, v_{t\sqrt{\xi}})$, which satisfies the process given by
    \begin{align*}
        \D \tilde x_t &= \xi \tilde{v}_t\, \D t\,,\\
        \D \tilde v_t &= -\xi \tilde \gamma \tilde{v}_t\, \D t - \nabla U(\tilde{x}_t)\, \D t + \sqrt{2\tilde \gamma}\, \D B_t\,,
    \end{align*}
    with $\tilde \gamma = 2$, $\xi = 2L$, which are the parameters satisfying~\cite[Proposition 1]{maetal2021nesterovmcmc}. From that Proposition, we know that the Lyapunov functional given by
    \begin{align*}
        \tilde{\eu F}(\tilde\mu' \mmid \tilde\mu) = \msf{KL}(\tilde\mu' \mmid \tilde\mu) +  \E_{\tilde\mu'}\bigl[\bigl\lVert \mf N^{1/2}\, \nabla \log \frac{\tilde\mu'}{\tilde\mu}\bigr\rVert^2 \bigr]\,, \qquad \text{where}~\mf N = \frac{1}{L}\,\begin{bmatrix} 1/4 & 1/2 \\ 1/2 & 2 \end{bmatrix} \otimes I_d\,,
    \end{align*}
    decays with $\partial_t \tilde{\eu F}(\tilde \mu_t \mmid \tilde \mu) \leq -\frac{1}{10 C_{\msf{LSI}}}\, \tilde{\eu F}(\tilde \mu_t \mmid \tilde\mu).$ Here the $\msf{LSI}$ constant does not change under our coordinate transform, but now $\tilde{\mu}_t$ represents the joint law of $(\tilde x_t, \tilde v_t)$, while the stationary measure has the form $\tilde \mu(\tilde x,\tilde v) \propto \pi(\tilde x) \times \exp(-\xi\, \norm{\tilde v}^2/2)$.
    The statement of our theorem immediately follows by reversing our change of variables, which involves scaling up the gradients of the momenta by $\xi^{1/2}$, while the time is scaled down by $\xi^{1/2}$.
\end{proof}

\subsection{Contraction of ULMC}

In this section, we prove a contraction result for ULMC and use this to deduce a log-Sobolev inequality along the trajectory of the underdamped Langevin diffusion.
The mean of the next iterate of ULMC started at $(x,v)$ is given by
\begin{align*}
    F(x,v)
    &\deq \Bigl(x + \frac{1-\exp(-\gamma h)}{\gamma}\, v - \frac{h - \gamma^{-1}\,(1-\exp(-\gamma h))}{\gamma} \, \nabla U(x), \\
    &\qquad\qquad\qquad\qquad\qquad\qquad \exp(-\gamma h)\,v - \frac{1-\exp(-\gamma h)}{\gamma}\,\nabla U(x) \Bigr)\,.
\end{align*}
We will use the change of coordinates
\begin{align*}
    (\phi,\psi)
    &\deq \mc M(x, v)
    \deq \bigl(x, x + \frac{2}{\gamma}\,v\bigr)\,.
\end{align*}
In these new coordinates, the mean of the next iterate of ULMC started at $(\phi,\psi)$ is $\bar F(\phi,\psi)$, where $\bar F = \mc M \circ F\circ \mc M^{-1}$.
Since $\mc M^{-1}(\phi,\psi) = (\phi, \frac{\gamma}{2}\,(\psi - \phi))$, we can explicitly write
\begin{align*}
    \bar F(\phi,\psi)
    &= \Bigl( \phi + \frac{1-\exp(-\gamma h)}{2}\,(\psi - \phi)- \frac{h - \gamma^{-1}\,(1-\exp(-\gamma h))}{\gamma}\,\nabla U(\phi), \\
    &\qquad \phi + \frac{1+\exp(-\gamma h)}{2}\,(\psi-\phi) - \frac{h+\gamma^{-1}\,(1-\exp(-\gamma h))}{\gamma}\,\nabla U(\phi)\Bigr)\,.
\end{align*}

\begin{lemma}
    Consider the mapping $\bar F : \R^d\times \R^d\to\R^d\times\R^d$ defined above.
    Assume that $mI_d \preceq \nabla^2 U \preceq LI_d$.
    Then, for $h \lesssim 1$ and $\gamma = c\sqrt{L}$ for some $c \geq \sqrt{2}$, $\bar F$ is a contraction with parameter
    \begin{align*}
        \norm{\bar F}_{\Lip}
        &\le 1 - \frac{m}{\sqrt{2L}}\,h + O(Lh^2)\,.
    \end{align*}
\end{lemma}
\begin{proof}
    We compute the partial derivatives
    \begin{align*}
        \partial_\phi {\bar F(\phi,\psi)}_\phi
        &= \frac{1+\exp(-\gamma h)}{2}\, I_d - \frac{h-\gamma^{-1}\,(1-\exp(-\gamma h))}{\gamma}\,\nabla^2 U(\phi)\,, \\
        \partial_\phi {\bar F(\phi,\psi)}_\psi
        &= \frac{1-\exp(-\gamma h)}{2}\, I_d - \frac{h+\gamma^{-1}\,(1-\exp(-\gamma h))}{\gamma}\,\nabla^2 U(\phi)\,, \\
        \partial_\psi {\bar F(\phi,\psi)}_\phi
        &= \frac{1-\exp(-\gamma h)}{2}\, I_d\,, \\
        \partial_\psi {\bar F(\phi,\psi)}_\psi
        &= \frac{1+\exp(-\gamma h)}{2} \, I_d\,.
    \end{align*}
    Let $a \deq \exp(-\gamma h)$ and $b \deq \frac{2}{\gamma}\,(h+\gamma^{-1}\,(1-\exp(-\gamma h)))$.
    Since
    \begin{align*}
        \frac{h-\gamma^{-1}\,(1-\exp(-\gamma h))}{\gamma} = O(h^2)\,,
    \end{align*}
    we have
    \begin{align*}
        \norm{\nabla \bar F(\phi,\psi)}_{\rm op}
        &\le \frac{1}{2}\, \Bigl\lVert \underbrace{\begin{bmatrix} (1+a)\,I_d & (1-a)\,I_d - b\,\nabla^2 U(\phi) \\ (1-a)\, I_d & (1+a)\, I_d \end{bmatrix}}_{\eqqcolon A}\Bigr\rVert_{\rm op} + O(L h^2)\,.
    \end{align*}
    Then,
    \begin{align*}
        AA^\T
        &= \begin{bmatrix} {(1+a)}^2\,I_d + {((1-a)\,I_d-b\,\nabla^2 U(\phi))}^2 & * \\ 2\,(1-a^2)\,I_d - (1+a)\,b\,\nabla^2 U(\phi) & \{{(1-a)}^2 + {(1+a)}^2\}\,I_d\end{bmatrix}\,,
    \end{align*}
    where the upper right entry is determined by symmetry.
    Since $1-a = \Theta(\gamma h)$ and $b = O(h/\gamma)$, one can simplify this as follows:
    \begin{align*}
        &\Bigl\lVert AA^\T - 2\,\underbrace{\begin{bmatrix} (1+a^2)\,I_d & (1-a^2)\,I_d - b\,\nabla^2 U(\phi) \\ (1-a^2)\, I_d - b\,\nabla^2 U(\phi) & (1+a^2)\, I_d \end{bmatrix}}_{\eqqcolon B}\Bigr\rVert_{\rm op} \\
                                                            &\qquad \le O\bigl(\frac{L^2 h^2}{\gamma^2} + Lh^2\bigr)\,.
    \end{align*}
    One can check that the eigenvalues of the matrix $B$ are $1+a^2 \pm (1-a^2 - b\lambda)$, where $\lambda$ ranges over the eigenvalues of $\nabla^2 U(\phi)$.
    Hence, we can bound
    \begin{align*}
        \norm B_{\rm op}
        &\le \max\{2a^2+Lb, 2-bm\}\,.
    \end{align*}
    We note that
    \begin{align*}
        2a^2 + Lb
        &= 2\exp(-2\gamma h) + \frac{2L\,(h+\gamma^{-1}\,(1-\exp(-\gamma h)))}{\gamma} \\
        &= 2\,\Bigl\{1 - 2\gamma h + \frac{2Lh}{\gamma} + O(\gamma^2 h^2 + Lh^2)\Bigr\}\,.
    \end{align*}
    In order for this to be strictly smaller than $2$, we must take $\gamma > \sqrt L$.
    We choose $\gamma = c\sqrt{L}$ for $c \geq \sqrt{2}$,
    in which case
    \begin{align*}
        \norm B_{\rm op}
        &\le 2 \max\Bigl\{1-c\sqrt{L}\,h,\; 1- m\sqrt{\frac{2}{L}}\,h\Bigr\} + O(Lh^2) \\
        &= 2\,\Bigl(1-m\sqrt{\frac{2}{L}}\,h\Bigr) + O(Lh^2)\,.
    \end{align*}
    We deduce that
    \begin{align*}
        \norm{AA^\T}_{\rm op}
        &\le 4\,\Bigl(1-m\sqrt{\frac{2}{L}}\,h\Bigr) + O(Lh^2)
    \end{align*}
    and therefore
    \begin{align*}
        \norm{\nabla \bar F(\phi,\psi)}_{\rm op}
        &\le \sqrt{1-m\sqrt{\frac{2}{L}}\,h} + O(Lh^2)
        \le 1-\frac{m}{\sqrt{2L}}\,h + O(Lh^2)\,. 
    \end{align*}
\end{proof}

The ULMC iterate is
\begin{align*}
    (x_{(k+1)h}, v_{(k+1)h})
    &\eqdist F(x_{kh}, v_{kh}) + \mc N(0, \Sigma)\,,
\end{align*}
where $\Sigma$ is the covariance of the Gaussian random vector in the LMC update.
In the new coordinates, this iteration can be written
\begin{align*}
    (\phi_{(k+1)h}, \psi_{(k+1)h})
    &\eqdist \bar F(\phi_{kh}, \psi_{kh}) + \mc N(0, \mc M \Sigma \mc M^\T)\,.
\end{align*}
Writing $\mc M \Sigma \mc M^\T = \bar \Sigma\otimes I_d$, we can compute
\begin{align*}
    \bar \Sigma_{1,1}
    &= \frac{2h}{\gamma} - \frac{3}{\gamma^2} + \frac{4\exp(-\gamma h)}{\gamma^2} - \frac{\exp(-2\gamma h)}{\gamma^2}
    = O(\gamma h^3)\,, \\
    \bar \Sigma_{1,2}
    &= \frac{2h}{\gamma} - \frac{1}{\gamma^2} + \frac{\exp(-2\gamma h)}{\gamma^2}
    = O(h^2)\,, \\
    \bar \Sigma_{2,2}
    &= \frac{2h}{\gamma} + \frac{5}{\gamma^2} - \frac{8\exp(-\gamma h)}{\gamma^2} + \frac{3\exp(-2\gamma h)}{\gamma^2}
    = \frac{4h}{\gamma^2} + O(h^2)\,.
\end{align*}
We conclude that
\begin{align*}
    \norm{\bar \Sigma}_{\rm op}
    &\le \frac{4h}{\gamma} + O(h^2)\,.
\end{align*}
Hence, $C_{\msf{LSI}}(\mc N(0, \mc M \Sigma \mc M^\T)) \le \frac{4h}{\gamma^2} + O(h^2)$.

\begin{proposition}
    Let $\hat\mu_t^{\mc M} \deq \law(\phi_t,\psi_t)$.
    Then, for all $\varepsilon > 0$, for all sufficiently small $h > 0$ (depending on $\varepsilon$), one has
    \begin{align*}
        C_{\msf{LSI}}(\hat\mu_{Nh}^{\mc M})
        &\le \Bigl(1 - \bigl(m\sqrt{\frac{2}{L}} - \varepsilon\bigr) \,h \Bigr){\vphantom{\Big|}}^N \, C_{\msf{LSI}}(\hat \mu_0^{\mc M})
        + \frac{4}{2m-\varepsilon\sqrt{2L}} + O\bigl( \frac{h\sqrt L}{m}\bigr)\,.
    \end{align*}
\end{proposition}
\begin{proof}
    The LSI constant evolves according to
    \begin{align*}
        C_{\msf{LSI}}(\hat \mu_{(k+1)h}^{\mc M})
        &\le \norm{\bar F}_{\rm op}^2 \, C_{\msf{LSI}}(\hat \mu_{kh}^{\mc M}) + C_{\msf{LSI}}\bigl(\mc N(0, \mc M \Sigma \mc M^\T)\bigr) \\
        &\le \Bigl(1 - m\sqrt{\frac{2}{L}} \,h + O(Lh^2)\Bigr) \, C_{\msf{LSI}}(\hat \mu_{kh}^{\mc M}) + \frac{4h}{\gamma} + O(h^2)\,.
    \end{align*}
    For $h$ sufficiently small, we have
    \begin{align*}
        C_{\msf{LSI}}(\hat \mu_{(k+1)h}^{\mc M})
        &\le \Bigl(1 - \bigl(m\sqrt{\frac{2}{L}} - \varepsilon\bigr) \,h \Bigr) \, C_{\msf{LSI}}(\hat \mu_{kh}^{\mc M}) + \frac{4h}{\gamma} + O(h^2)\,.
    \end{align*}
    Iterating,
    \begin{align*}
        C_{\msf{LSI}}(\hat \mu_{Nh}^{\mc M})
        &\le \Bigl(1 - \bigl(m\sqrt{\frac{2}{L}} - \varepsilon\bigr) \,h \Bigr){\vphantom{\Big|}}^N \, C_{\msf{LSI}}(\hat \mu_0^{\mc M})
        + \frac{4}{2m-\varepsilon\sqrt{2L}} + O\bigl( \frac{h\sqrt L}{m}\bigr)\,.
    \end{align*}
    This completes the proof.
\end{proof}

\begin{corollary}
    Let $\mu_t^{\mc M}$ now denote the law of the continuous-time underdamped Langevin diffusion with $\gamma = c\sqrt{L}$ for $c \geq \sqrt{2}$ in the $(\phi,\psi)$ coordinates.
    Then,
    \begin{align*}
        C_{\msf{LSI}}(\mu_t^{\mc M})
        &\le \exp\Bigl( - m\sqrt{\frac{2}{L}}\,t\Bigr)\, C_{\msf{LSI}}(\mu_0^{\mc M}) + \frac{2}{m}\,.
    \end{align*}
\end{corollary}
\begin{proof}
    In the preceding proposition, let $h\searrow 0$ while $Nh\to t$, and then let $\varepsilon \searrow 0$.
\end{proof}

\section{Discretization Analysis}

We consider the discretization used in~\cite{maetal2021nesterovmcmc}, with the following differential form:
\begin{align*}
    &d\hat{x}_t =  \hat v_t \,\D t\,,\\
    &d\hat{v}_t = -\gamma  \hat v_t \,\D t - \nabla U(\hat x_{kh})\,\D t + \sqrt{2\gamma} \,\D B_t\,,
\end{align*}
and we define the variable $\hat{w}_t$ as the tuple $(\hat x_t, \hat v_t)$, for $t \in [kh, (k+1)h]$.

\subsection{Technical Lemmas}

\begin{theorem}[{Girsanov's Theorem, Adapted from~\cite[Theorem 8.6.8]{oksendal2013stochastic}}]\label{thm:girsanov}
    Consider stochastic processes ${(x_t)}_{t\ge 0}$, ${(b_t^P)}_{t\ge 0}$, ${(b_t^Q)}_{t\ge 0}$ adapted to the same filtration, and $\sigma \in \mathbb{R}^{d \times d}$ any constant, possibly degenerate, matrix. Let $P_T$ and $Q_T$ be probability measures on the path space $C([0,T]; \R^d)$ such that ${(w_t)}_{t\ge 0}$ evolves according to
    \begin{align*}
        \D w_t
        &= b_t^P \, \D t + \sigma \, \D B_t^P \qquad\text{under}~P_T\,, \\
        \D w_t
        &= b_t^Q \, \D t + \sigma \, \D B_t^Q \qquad\text{under}~Q_T\,,
    \end{align*}
    where $B^P$ is a $P_T$-Brownian motion and $B^Q$ is a $Q_T$-Brownian motion.
    Furthermore, suppose there exists a process $(u_t)_{t\ge 0}$ such that 
    \begin{equation*}
        \sigma \, u_t = b^P_t - b^Q_t \,, 
    \end{equation*}
    and 
    \begin{equation*}
        \mathbb{E}^{Q_T} \exp\Bigl( 2q^2 \int_0^T
        \| u_s \|^2 \, \D s 
        \Bigr) < \infty \,,
    \end{equation*}
    Consequently, if we define $\sigma^{\dagger}$ as the Moore--Penrose pseudo-inverse of $\sigma$, then by the previous supposition we have $u_t = \sigma^\dagger\, (b_t^P - b_t^Q)$.
    Then,
    \begin{align*}
        \frac{\D P_T}{\D Q_T}
        &= \exp\Bigl(\int_0^T \langle \sigma^\dagger\,(b_t^P - b_t^Q), \D B_t^Q \rangle - \frac{1}{2} \int_0^T \norm{\sigma^\dagger\,(b_t^P - b_t^Q)}^2 \, \D t\Bigr)\,.
    \end{align*}
\end{theorem}

In fact, we will only need the following corollary.
\begin{corollary}\label{cor:girsanov}
    For any event $\mc E$ and $q \ge 1$,
    \begin{align*}
        \E^{Q_T}\bigl[\bigl( \frac{\D P_T}{\D Q_T} \bigr)^q \one_{\mc E} \bigr]
        &\le \sqrt{\E\Bigl[\exp\Bigl(2q^2 \int_0^T \norm{\sigma^\dagger\,(b_t^P - b_t^Q)}^2 \, \D t\Bigr) \one_{\mc E}\Bigr]}\,.
    \end{align*}
\end{corollary}
\begin{proof}
    Using Cauchy--Schwarz, and then It\^o's Lemma, we find
    \begin{align*}
        &\E^{Q_T}\bigl[\bigl( \frac{\D P_T}{\D Q_T} \bigr)^q \one_{\mc E} \bigr] = \E^{Q_T} \Bigl[\exp\Bigl(q \int_0^T \langle \sigma^\dagger\,(b_t^P - b_t^Q), \D B_t^Q \rangle - \frac{q}{2} \int_0^T \norm{\sigma^\dagger\,(b_t^P - b_t^Q)}^2 \, \D t\Bigr)\one_{\mc E}\Bigr] \\
        &\qquad \leq \sqrt{\E^{Q_T} \Bigl[\exp\Bigl((2q^2 - q) \int_0^T \norm{\sigma^\dagger\,(b_t^P - b_t^Q)}^2 \, \D t\Bigr)\one_{\mc E}\Bigr]} \\
        &\qquad\qquad{} \times \underset{= 1}{\underbrace{\sqrt{\E^{Q_T} \Bigl[\exp\Bigl(2q \int_0^T \langle \sigma^\dagger\,(b_t^P - b_t^Q), \D B_t^Q \rangle - 2q^2 \int_0^T \norm{\sigma^\dagger\,(b_t^P - b_t^Q)}^2 \, \D t\Bigr)\one_{\mc E}\Bigr]}}} \\
        &\qquad \leq \sqrt{\E^{Q_T} \Bigl[\exp\Bigl(2q^2 \int_0^T \norm{\sigma^\dagger\,(b_t^P - b_t^Q)}^2 \, \D t\Bigr)\one_{\mc E}\Bigr]}\,.
    \end{align*}
    Here, we used the fact that $t\mapsto \exp(\int_0^t \langle u_\tau, \D B_\tau\rangle - \frac{1}{2} \int_0^t \norm{u_\tau}^2 \, \D \tau)$ is a local martingale.
\end{proof}

We can identify the following for the process $(x_t, v_t)$:
\begin{align*}
    \sigma = \begin{bmatrix}
    0 &  0 \\
    0 & \sqrt{2\gamma}\,I_d
    \end{bmatrix}\,, \qquad  b_t^P = \begin{bmatrix}
    v_t\\
    -\gamma v_t - \nabla U(x_{t})
    \end{bmatrix}\,, \qquad 
    b_t^Q = \begin{bmatrix}
    v_t \\
    -\gamma v_t - \nabla U(x_{\floor{t/h} h})
    \end{bmatrix}\,.
\end{align*}
In this case, $\norm{\sigma^\dagger\,(b_t^P - b_t^Q)} \equiv \frac{1}{\sqrt{2\gamma}}\, \norm{\nabla U(x_{\floor{t/h}h})-\nabla U(x_t)}$.

We also adapt the following Lemmas without proof from~\cite{chewi2021analysis}.

\begin{lemma}[{Change of Measure, from \cite[Lemma~21]{chewi2021analysis}}]\label{lem:change_of_measure}
    Let $\mu$, $\nu$ be probability measures and let $E$ be any event.
    Then,
    \begin{align*}
        \mu(E)
        &\le \nu(E) + \sqrt{\chi_2(\mu \mmid \nu) \, \nu(E)}\,.
    \end{align*}
    In particular, if $\mu$ and $\nu$ are probability measures on $\R^d$ and
    \begin{align*}
        \nu\{\norm \cdot \ge R_0 + \eta\} \le C\exp(-c\eta^2) \qquad\text{for all}~\eta \ge 0\,,
    \end{align*}
    where $C \ge 1$, then
    \begin{align*}
        \mu\Bigl\{\norm \cdot \ge R_0 + \sqrt{\frac{1}{c} \, \eu R_2(\mu\mmid \nu)} + \eta\Bigr\} \le 2C\exp\bigl(-\frac{c\eta^2}{2}\bigr) \qquad\text{for all}~\eta \ge 0\,.
    \end{align*}
\end{lemma}

\begin{lemma}\label{lem:brownian_mgf}
    Let ${(B_t)}_{t\ge 0}$ be a standard Brownian motion in $\R^d$.
    Then, if $\lambda \ge 0$ and $h \le 1/(4\lambda)$,
    \begin{align*}
        \E \exp\bigl(\lambda \sup_{t\in [0,h]}{\norm{B_t}^2}\bigr)
        &\le \exp(6dh\lambda)\,.
    \end{align*}
    In particular, for all $\eta \ge 0$,
    \begin{align*}
        \Pr\bigl\{\sup_{t \in [0,h]}{\norm{B_t}} \ge \eta\bigr\}
        &\le 3\exp\bigl(-\frac{\eta^2}{6dh}\bigr)\,.
    \end{align*}
\end{lemma}

\begin{lemma}[{\cite[Lemma~14]{ganeshtalwar2020renyi}}]
\label{lem:uncondition}
Let $Y>0$ be a random variable.
Assume that for all $0 < \delta < 1/2$ there exists an event $\mathcal{E}_\delta$ with probability at least $1-\delta$ such that
\begin{align*}
    \E[Y^2 \mid \mathcal{E}_\delta] \leq \frac{v}{\delta^\xi}
\end{align*}
for some $\xi < 1$.
Then, $\E Y \leq 4 \sqrt{v}$.
\end{lemma}

\begin{lemma}[Matrix Gr\"onwall Inequality]\label{lem:matrix_gronwall}
    Let $x: \R_+ \to \R^d$, and $c \in \R^d$, $A \in \R^{d \times d}$, where $A$ has non-negative entries. Suppose that the following inequality is satisfied componentwise:
    \begin{align}
        x(t) \leq c + \int_0^t Ax(s) \, \D s\,, \qquad\text{for all}~t\ge 0\,.
    \end{align}
    Then, the following inequality holds, where $I_d \in \R^{d \times d}$ is the $d$-dimensional identity matrix:
    \begin{align}
        x(t) \leq (AA^{\dagger}\, e^{At} - AA^{\dagger} + I_d)\, c\,,
    \end{align}
    where $A^{\dagger}$ is the Moore--Penrose pseudo-inverse of $A$ (when $A$ is invertible, this is equivalent to the standard inverse).
\end{lemma}

\begin{proof}
    This is a special case of~\cite[Main Theorem]{chandra1976linear}.
\end{proof}

\subsection{Movement Bound for ULMC}

We next prove a movement bound for the continuous-time Langevin diffusion. The following lemma is a standard fact about the concentration of the norm of a Gaussian vector~\cite[see, e.g.,][Theorem 5.5]{boucheronlugosimassart2013concentration}.

\begin{lemma}[Concentration of the Norm] \label{lem:concentration_target_lsi}
    The following concentration holds: for all $\eta \ge 0$,
    \begin{align*}
        &\rho(\norm{\cdot} \geq \sqrt{d} + \eta) \leq \exp\Bigl(-\frac{\eta^2}{2}\Bigr)\,.
    \end{align*}
\end{lemma}

Note that $\norm{v_t - v_0}$ is of size $\mc O(\sqrt{dt})$, due to the Brownian motion component of the momentum variable $v$; this is the same order as the size of the increment of the overdamped Langevin diffusion. However, if we consider the increment in the $x$-coordinate only, we obtain the following bound. 

\begin{lemma}\label{lem:better_stoc_calc}
    Let ${(x_t, v_t)}_{t\ge 0}$ denote the continuous-time underdamped Langevin diffusion started at $(x_0, v_0)$, and assume that the gradient $\nabla U$ of the potential satisfies $\nabla U(0) = 0$ and is H\"older continuous (satisfies~\eqref{eq:holder}).
    Also, assume that $h \lesssim L^{-1/2} \wedge \gamma^{-1}$ and $0 \le \lambda \lesssim \frac{1}{\gamma^s d^s h^{3s}}$.
    Then,
    \begin{align*}
        \log \E\exp\bigl(\lambda\sup_{t\in [0,h]}{\norm{x_t - x_0}}^{2s}\bigr)
        \lesssim \bigl(L^{2s} h^{4s} \,(1+\norm{x_0}^{2s^2}) + h^{2s}\, \norm{v_{0}}^{2s} + \gamma^s d^s h^{3s}\bigr)\, \lambda\,.
    \end{align*}
\end{lemma}

\begin{proof}
    For the interpolant times, we will use Gr\"onwall's matrix inequality (Lemma~\ref{lem:matrix_gronwall}), with the following equation for $x$: 
    \begin{align*}
        \norm{x_{t} - x_{0}}
        &\le \Bigl\lVert \int_0^t v_\tau \, \D \tau \Bigr\rVert
        \le h\,\norm{v_0} + \Bigl\lVert \int_0^t (v_\tau - v_0) \, \D \tau\Bigr\rVert \\
        &\leq h\,\norm{v_0} + \Bigl\lVert\int_0^t \int_0^\tau \gamma v_{\tau'} \, \D \tau' \, \D \tau \Bigr\rVert + \Bigl\lVert \int_0^t \int_0^\tau \nabla U(x_{\tau'})\, \D \tau' \, \D \tau \Bigr\rVert \\
        &\qquad{} + \Bigl\lVert \int_0^t \int_0^\tau \sqrt{2\gamma} \, \D B_{\tau'} \, \D \tau\Bigr\rVert \\
        &\le h\,\norm{v_0} + \gamma h\,\Bigl(h\,\norm{v_0} + \int_0^t \norm{v_{\tau} -v_{0}} \, \D \tau\Bigr) + Lh^2  \\
        &\qquad{} + Lh\,\Bigl(h\,\norm{x_{0}}^s+ \int_0^t \norm{x_{\tau} -x_{0}} \, \D \tau \Bigr) + \sqrt{2\gamma}\, h \sup_{t \in [0,h]}{\norm{B_{t}}}\,.
    \end{align*}
    Here we use the H\"older property of $\nabla U$ along with $\norm{x}^s \leq 1+\norm{x}$. 
    Likewise for $v$:
    \begin{align*}
        \norm{v_{t} - v_{0}} &\leq \Bigl\lVert \int_0^t \gamma v_\tau \, \D \tau\Bigr\rVert + \Bigl\lVert\int_0^t \nabla U(x_\tau)\, \D \tau \Bigr\rVert +  \Bigl\lVert\int_0^t \sqrt{2\gamma} \, \D B_\tau\Bigr\rVert \\
        &\leq \gamma\,\Bigl(h\, \norm{v_0} + \int_0^t \norm{v_\tau -v_0} \, \D \tau\Bigr) + Lh + L\,\Bigl(h\,\norm{x_0}^s + \int_0^t \norm{x_{\tau} -x_{0}} \, \D \tau \Bigr) \\
        &\qquad{} + \sqrt{2\gamma} \sup_{t \in [0,h]} {\norm{B_{t}}}\,.
    \end{align*}
    
    Consequently, we can use the matrix form of Gr\"onwall's inequality (Lemma \ref{lem:matrix_gronwall}). While applying that Lemma, let $c = c_1 + c_2$ with $c_1, c_2$ to be given. First, for $c_1$:
    \begin{align*}
        A = \begin{bmatrix}
            Lh & \gamma h  \\
            L & \gamma
        \end{bmatrix}\,, \qquad c_1 = \begin{bmatrix}
            Lh^2\, \norm{x_0}^s + \gamma h^2\, \norm{v_0} + Lh^2+ \sqrt{2\gamma}\, h\,\sup_{t \in [0, h]} \norm{B_{t}}  \\
            Lh\, \norm{x_0}^s + \gamma h\, \norm{v_0} + Lh +  \sqrt{2\gamma} \sup_{t \in [0, h]}\norm{B_{t}}
        \end{bmatrix}\,.
    \end{align*}
    Noting that $c_1$ lies in the image space of $A$ so that $AA^\dagger c_1= c_1$, and similarly observing that $\exp(At) \,c_1$ belongs to the image space of $A$ (using the power series representation of the matrix exponential), we obtain for this first component:
 \begin{align*}
        &\sup_{t \in [0, h]}{\norm{x_{0} - x_{t}}} \\
        &\qquad \leq h \exp\bigl((Lh+\gamma)\,h\bigr)\, \bigl(\gamma h\, \norm{v_{0}} + Lh\, \norm{x_{0}}^s + Lh + \sqrt{2\gamma} \sup_{t \in [0,h]}{\norm{B_{t}}}\bigr) + c_2~\text{term} \\
        &\qquad \leq 2h \,\bigl(\gamma h\, \norm{v_{0}} + Lh\, \norm{x_{0}}^s + Lh + \sqrt{2\gamma} \sup_{t \in [0,h]}{\norm{B_{t}}}\bigr) + c_2~\text{term}\,,
    \end{align*}
    where in the second line we take $h \lesssim \frac{1}{\sqrt L+\gamma}$.
    Now, taking
    \begin{align*}
        c_2 = \begin{bmatrix}
            h\,\norm{v_{0}}  \\
            0 
            \end{bmatrix}\,,
    \end{align*}
    we find the following (where $\mathbf{v}_{(1)}$ denotes the first component of a vector $\mathbf{v}$):
    \begin{align*}
        ((AA^\dagger\, (e^{Ah} - I_{2d}) + I_{2d})\, c_2)_{(1)} = \frac{Lhe^{(Lh+\gamma)\,h} + \gamma}{Lh + \gamma}\,h\,\norm{v_{0}}.
    \end{align*}
    Finally, for $h \lesssim \frac{1}{\sqrt L + \gamma}$, this can be bounded by $2h\,\norm{v_{0}}$. Using Lemma~\ref{lem:brownian_mgf} and plugging this into the expression completes the proof.
\end{proof}

\subsection{Sub-Gaussianity of the Iterates}\label{scn:subgaussianty}

Similarly to~\cite{chewi2021analysis}, we introduce a modified potential in order to prove sub-Gaussianity of the iterates of ULMC\@.
Firstly, we consider a modified distribution in the $x$-coordinate, with parameter $a \defeq (\beta, S)$ for some $S, \beta \geq 0$: 
\begin{align}\label{eq:modified_potential}
\pi^{(a)} \propto \exp(-U^{(a)})\,, \qquad U^{(a)}(x) \defeq U(x) + \frac{\beta}{2}\, (\norm{x} - S)^2_+\,.
\end{align}

The modified potential satisfies the following properties.

\begin{lemma}[{Properties of the Modified Potential,~\cite[Lemma~23]{chewi2021analysis}}]
    \label{lem:modified} 
    Consider $\pi^{(a)}$ and $U^{(a)}$ defined as in~\eqref{eq:modified_potential}. Assume that $\nabla U(0) = 0$ and that $\nabla U$ satisfies~\eqref{eq:holder}.
    Then, the following assertions hold.
    \begin{enumerate}
        \item (sub-Gaussian tail bound) Assume that $S$ is chosen so that $\pi(B(0,S)) \ge 1/2$.
         Then, for all $\eta \ge 0$,
         \begin{align*}
             \pi^{(a)}\{\norm \cdot \ge S + \eta\}
            \le 2\exp\bigl(-\frac{\beta \eta^2}{2}\bigr)\,.
         \end{align*}
         \item (gradient growth) The gradient $\nabla U^{(a)}$ satisfies
         \begin{align*}
             \norm{\nabla U^{(a)}(x)}
             &\le L + (\beta+L) \, \norm x\,.
         \end{align*}
    \end{enumerate}
\end{lemma}

Then, letting $\{(x^{(a)}_t, v^{(a)}_t) \}_{t \geq 0}$ be the solution to the underdamped Langevin diffusion with potential $U^{(a)}$ and $\mu^{(a)} \deq \pi^{(a)}\otimes \rho$, the following lemma holds:

\begin{lemma}\label{lem:altered_subgsn}
    Assume that $h \lesssim (\beta+L)^{-1/2} \wedge \gamma^{-1} \wedge d^{-1/2}$, and $\beta\le 1$.
    Then, for all $\delta \in (0,1)$, with probability at least $1-\delta$,
    \begin{align*}
        \sup_{t \leq Nh}{\norm{x_t^{(a)}}} - S
        &\lesssim (\beta + L)\, Sh^2 + \sqrt{\frac{1}{\beta}\,\eu R_2(\mu_0^{(a)} \mmid \mu^{(a)})} + \sqrt{\frac{1}{\beta}\log\frac{16N}{\delta}}\,.
    \end{align*}
\end{lemma}
\begin{proof}
   We can use the change of measure lemma (Lemma~\ref{lem:change_of_measure}) together with the sub-Gaussian tail bounds in Lemmas~\ref{lem:concentration_target_lsi}, \ref{lem:modified} to see that with probability at least $1-\delta$, the following events hold simultaneously:
    \begin{align*}
        \max_{k \leq N}{\norm{x^{(a)}_{kh}}}
        &\le S + \sqrt{\frac{2}{\beta} \, \eu R_2(\mu_{0}^{(a)} \mmid \mu^{(a)})} + \sqrt{\frac{4}{\beta} \log \frac{8N}{\delta}}\, \\
        \max_{k \leq N}{\norm{v_{kh}^{(a)}}}
        &\le \sqrt{d} + \sqrt{2 \, \eu R_2( \mu_{0}^{(a)} \mmid \mu^{(a)})} + \sqrt{4 \log \frac{4N}{\delta}}\, .
    \end{align*}
    Here we use a union bound together with the monotonicity of $t\mapsto \eu R_2(\mu_t^{(a)} \mmid \mu^{(a)})$ in $t$.
    
    For the interpolant times, we will use Gr\"onwall's matrix inequality, with the following inequality for $x$:
    \begin{align*}
        \norm{x^{(a)}_{kh} - x^{(a)}_{kh+t}} &\leq h\, \norm{v^{(a)}_{kh}} + \Bigl \lVert \int_0^t \int_0^\tau \gamma v^{(a)}_{kh+\tau'} \, \D \tau' \, \D \tau \Bigr \rVert + \Bigl \lVert\int_0^t \int_0^\tau \nabla U^{(a)}(x^{(a)}_{kh+\tau'})\, \D \tau' \, \D \tau\Bigr \rVert \\
        &\qquad +  \Bigl \lVert \int_0^t \int_0^\tau \sqrt{2\gamma} \,  \D B_{kh+\tau'} \, \D \tau \Bigr \rVert \\
        &\leq h\, \norm{v^{(a)}_{kh}} + \gamma h\,\Bigl(h\,\norm{v^{(a)}_{kh}} + \int_0^t \norm{v^{(a)}_{kh+\tau} -v^{(a)}_{kh}} \, \D \tau\Bigr) + L h^2 \\
        &\qquad{} + (\beta+L)\,h\,\Bigl(h\,\norm{x^{(a)}_{kh}} + \int_0^t \norm{x^{(a)}_{kh+\tau} -x^{(a)}_{kh}} \, \D \tau \Bigr) \\
        &\qquad{} + \sqrt{2\gamma}\, h \sup_{\tau \in [0, h]}{\norm{B_{kh+\tau} - B_{kh}}}\,.
    \end{align*}
    Likewise,
    \begin{align*}
        \norm{v^{(a)}_{kh} - v^{(a)}_{kh+t}}
        &\leq \Bigl \lVert \int_0^t \gamma v^{(a)}_{kh+\tau} \, \D \tau \Bigr \rVert + \Bigl \lVert \int_0^t \nabla U^{(a)}(x^{(a)}_{kh+\tau})\, \D \tau \Bigr \rVert +  \Bigl \lVert \int_0^t \sqrt{2\gamma} \,  \D B_{kh+\tau}\Bigr \rVert \\
        &\leq \gamma\,\Bigl(h\, \norm{v^{(a)}_{kh}} + \int_0^t \norm{v^{(a)}_{kh+\tau} -v^{(a)}_{kh}} \, \D \tau\Bigr) + Lh\\
        &\qquad{} + (\beta+L)\,\Bigl(h\,\norm{x^{(a)}_{kh}} + \int_0^t \norm{x^{(a)}_{kh+\tau} -x^{(a)}_{kh}} \, \D \tau \Bigr) \\
        &\qquad{} + \sqrt{2\gamma}\sup_{\tau \in [0, h]} {\norm{B_{kh+\tau} - B_{kh}}}\,.
    \end{align*}
    
    Consequently, we can apply the matrix Gr\"onwall inequality analogously to how we did in Lemma~\ref{lem:matrix_gronwall} with $c = c_1 + c_2$ denoting the following matrices:
    \begin{align*}
        A &= \begin{bmatrix}
            (\beta + L)\,h & \gamma h  \\
            (\beta+L) & \gamma
        \end{bmatrix}\,, \\
        c_1 &=  \begin{bmatrix}
            (\beta+L)\,h^2\, \norm{x^{(a)}_{kh}} + \gamma h^2\, \norm{v^{(a)}_{kh}} + Lh^2+ \sqrt{2\gamma}\, h\sup_{t \in [0, h]}\norm{B_{kh+t} - B_{kh}}  \\
            (\beta+L)\,h\, \norm{x^{(a)}_{kh}} + \gamma h\, \norm{v^{(a)}_{kh}} + Lh +  \sqrt{2\gamma} \sup_{t \in [0, h]}\norm{B_{kh+t} - B_{kh}}
        \end{bmatrix}\,, \\
        c_2 &= \begin{bmatrix}
            h\, \norm{v^{(a)}_{kh}} \\
            0
        \end{bmatrix}\,.
    \end{align*}
    
    Note that $c_1$ here is again in the image space of $A$, so that $(AA^{\dagger} - I_2)\,c = 0$. Finally, after calculating the matrix exponential we find
    \begin{align*}
        \sup_{t \leq h}{\norm{x^{(a)}_{kh} - x^{(a)}_{kh+t}}}
        &\leq h \exp\bigl((\beta+L)\,h^2+ \gamma h\bigr)\, \Bigl((\beta+L)\, h\, \norm{x^{(a)}_{kh}} + \gamma h\, \norm{v^{(a)}_{kh}} + Lh \\
        &\qquad\qquad\qquad\qquad\qquad\qquad\qquad\qquad{} + \sqrt{2\gamma} \sup_{t \leq h} {\norm{B_{kh+t} - B_{kh}}}\Bigr)  \\
        &\qquad{} + h\, \frac{(\beta+L)\exp\bigl((\beta+L)\,h^2+ \gamma h\bigr) h  + \gamma}{(\beta+L)\,h + \gamma}\, \norm{v^{(a)}_{kh}} \\
        &\le 2h\, \Bigl((\beta+L)\,h\, \norm{x^{(a)}_{kh}} +\norm{v^{(a)}_{kh}} + Lh + \sqrt{2\gamma} \sup_{t \leq h}{\norm{B_{kh+t} - B_{kh}}}\Bigr)\,,
    \end{align*}
    where in the second line we take $h \lesssim \frac{1}{(\beta+L)^{1/2}} \wedge \frac{1}{\gamma}$. Note that this is also entirely analogous to the calculation in Lemma \ref{lem:better_stoc_calc}.
    
    Subsequently, we can take a union bound to obtain for any $S_1, S_2$,
    \begin{align*}
        &\P\Bigl\{\sup_{t \in [0, Nh]}{\norm{x^{(a)}_{t}}} \geq \eta \Bigr\} \\
        &\qquad \leq \P\Bigl\{\max_{k = 0, 1, \ldots N-1}{\norm{x^{(a)}_{kh}}} \geq S_1 \Bigr\} + \P\Bigl\{\max_{k = 0, 1, \ldots N-1}{\norm{v^{(a)}_{kh}}} \geq S_2 \Bigr\}  \\
        &\qquad\qquad{} + \sum_{k=0}^{N-1} \P\Bigl\{\sup_{t \in [0,h]}{\norm{x_{kh+t}^{(a)} - x_{kh}^{(a)}}} \ge \eta - S_1\Bigr\} \\
        &\qquad \leq \P\Bigl\{\max_{k = 0, 1, \ldots N-1}{\norm{x^{(a)}_{kh}}} \geq S_1 \Bigr\} + \P\Bigl\{\max_{k = 0, 1, \ldots N-1}{\norm{v^{(a)}_{kh}}} \geq S_2 \Bigr\}  \\
        &\qquad\qquad{} + \sum_{k=0}^{N-1} \P\Bigl\{\sup_{t \in [0,h]} \sqrt{2\gamma}\,\norm{B_{kh+t} - B_{kh}} \ge \frac{\eta - S_1}{2h}
        - (\beta+L)\,S_1 h - S_2 - L h \Bigr\}\,.
    \end{align*}
    
    Subsequently, taking respectively $S_1 = S + \sqrt{\frac{2}{\beta} \, \eu R_2(\mu_{0}^{(a)} \mmid \mu^{(a)})} + \sqrt{\frac{4}{\beta} \log \frac{8N}{\delta}}$, $S_2 = \sqrt{d} + \sqrt{2\,\eu R_2(\mu_{0}^{(a)} \mmid \mu^{(a)})} + \sqrt{4\log \frac{4N}{\delta}}$, we use the Brownian motion tail bound (Lemma \ref{lem:brownian_mgf}) to get with probability $1-2\delta$:
    \begin{align*}
        \sup_{t \leq Nh}{\norm{x_t^{(a)}}} - S_1
        &\lesssim (\beta+L)\,S_1 h^2 + S_2 h  + Lh^2 + \sqrt{\gamma dh^3 \log \frac{3N}{\delta}}\,.
    \end{align*}
    If we assume that $\beta\le 1$ and $h \lesssim \frac{1}{\sqrt d}$, then we can further simplify this bound to yield
    \begin{align*}
        \sup_{t \leq Nh}{\norm{x_t^{(a)}}} - S
        &\lesssim (\beta + L)\, Sh^2 + \sqrt{\frac{1}{\beta}\,\eu R_2(\mu_0^{(a)} \mmid \mu^{(a)})} + \sqrt{\frac{1}{\beta}\log\frac{8N}{\delta}}\,.
    \end{align*}
    This concludes the proof.
\end{proof}

To transfer this sub-Gaussianity to the original underdamped Langevin process, we consider the following bound on the chi-squared divergence between these two processes.

\begin{proposition}\label{prop:girs_altered}
    Let $Q_T, Q_T^{(a)}$ represent respectively the laws on the path space of the original and modified diffusions, under the same initialization $\mu_0$. Then, if $\beta \lesssim \frac{\gamma}{T} \wedge L$ and $h \lesssim (\beta + L)^{-1/2} \wedge \gamma^{-1} \wedge d^{-1/2}$, then
    \begin{align*}
        \eu R_2(Q_T \mmid Q_T^{(a)})
        \lesssim \frac{\beta^2 L^2 S^2 T h^4}{\gamma} + \frac{\beta T}{\gamma} \, \bigl(\eu R_2(\mu_0 \mmid \mu^{(a)}) + \log N\bigr)\,.
    \end{align*}
\end{proposition}
\begin{proof}
    Conditioning on the event in Lemma \ref{lem:altered_subgsn}, which we denote by $\mc E_{\delta}$ for some $\delta \leq 1/2$, then using Girsanov's theorem (Corollary \ref{cor:girsanov}) we get (for some sufficiently small $h$ so that Novikov's condition is satisfied)
    \begin{align*}
        \log \E\Bigl[ \Bigl(\frac{\D Q_T}{\D Q_T^{(a)}}\Bigr)^4\, \mathbbm{1}_{\mc E_\delta} \Bigr] &\leq \frac{1}{2} \log \E \Bigl[\exp \Bigl(\frac{16}{\gamma} \int_0^T \norm{\nabla U(x_t^{(a)}) - \nabla U^{(a)}(x_t^{(a)})}^2 \, \D t \Bigr)\, \mathbbm{1}_{\mc E_\delta} \Bigr] \\
        &\le \frac{1}{2} \log \E\Bigl[\exp \Bigl( \frac{16\beta^2}{\gamma} \int_0^T \bigl(\norm{x_t^{(a)}} - S\bigr)^2_+ \, \D t\Bigr)\, \ind_{\mc E_\delta} \Bigr] \\
        &\lesssim \frac{\beta^2 T}{\gamma} \, \Bigl\{ (\beta + L)^2\, S^2 h^4 + \frac{1}{\beta}\,\eu R_2(\mu_0 \mmid \mu^{(a)}) + \frac{1}{\beta} \log \frac{16N}{\delta}\Bigr\}\,.
    \end{align*}
If we take $\beta \lesssim \gamma/T$ and that $L\ge \beta$, we can use Lemma \ref{lem:uncondition} to get
\begin{align*}
    \eu R_2(Q_T \mmid Q_T^{(a)}) &= \log \E \Bigl[\Bigl( \frac{\D Q_T}{\D Q_T^{(a)}}\Bigr)^2 \Bigr]
   \lesssim \frac{\beta^2 L^2 S^2 T h^4}{\gamma} + \frac{\beta T}{\gamma} \, \bigl(\eu R_2(\mu_0 \mmid \mu^{(a)}) + \log N\bigr)\,.
\end{align*}
This concludes the proof.
\end{proof}

\begin{proposition}\label{prop:high_prob_diffusion}
    Consider the continuous time diffusion $(x_t, v_t)_{t\ge 0}$ initialized at $\mu_0$. For $h \lesssim (\beta+L)^{-1/2} \wedge \gamma^{-1} \wedge d^{-1/2}$, $S \asymp \mf m$, and $\beta \asymp \frac{\gamma}{T}$, for $\delta \in (0, 1/2)$, the following holds with probability $1-\delta$:
    \begin{align*}
        \max_{k = 0, \dotsc, N-1}{\norm{x_{kh}}}
        &\lesssim \mf m + \sqrt{\frac{T}{\gamma} \,\bigl(\eu R_2(\mu_0\mmid \mu^{(a)}) + \log \frac{N}{\delta}\bigr)}\,, \\
        \max_{k=0,\dotsc,N-1}{\norm{v_{kh}}}
        &\lesssim \sqrt d + \sqrt{\eu R_2(\mu_0\mmid \mu^{(a)}) + \log \frac{N}{\delta}}\,.
    \end{align*}
\end{proposition}
\begin{proof}
    Recall from the proof of Lemma \ref{lem:altered_subgsn} that with probability $1-\delta$,
    \begin{align*}
        \max_{k=0, \ldots N-1}{\norm{x_{kh}^{(a)}}} \lesssim S + \sqrt{\frac{1}{\beta}\, \eu R_2(\mu_0 \mmid \mu^{(a)})} + \sqrt{\frac{1}{\beta} \log \frac{8N}{\delta}}\,.
    \end{align*}
    In particular, this immediately implies that the following holds: for $\eta \ge 0$,
    \begin{align*}
        \Pr\Bigl(\max_{k=0, \ldots N-1}{\norm{x_{kh}^{(a)}}} \gtrsim S + \sqrt{\frac{1}{\beta}\, \eu R_2(\mu_0 \mmid \mu^{(a)})} + \sqrt{\frac{1}{\beta} \log \frac{8N}{\delta}} + \eta \Bigr)
        &\lesssim N \exp(-c\beta \eta^2)\,,
    \end{align*}
    for a universal constant $c > 0$.
    
    Then, using the change of measure (Lemma~\ref{lem:change_of_measure}) together with the bound in Proposition~\ref{prop:girs_altered}, choosing $S \asymp \mf m$, we get with probability $1-\delta$
    \begin{align*}
        \max_{k = 0, \ldots N-1}{\norm{x_{kh}}} &\lesssim S + \sqrt{\frac{1}{\beta}\, \eu R_2(\mu_0 \mmid \mu^{(a)})} + \sqrt{\frac{1}{\beta}\, \eu R_2(Q_T \mmid Q_T^{(a)})} + \sqrt{\frac{1}{\beta} \log \frac{N}{\delta}} \\
        &\lesssim \mf m + \sqrt{\frac{1}{\beta} \,\bigl(\eu R_2(\mu_0\mmid \mu^{(a)}) + \log \frac{N}{\delta}\bigr) + \frac{\beta L^2 T h^4 \mf m^2}{\gamma}}\,.
    \end{align*}
    We choose $\beta \asymp \gamma/T$ so that
    \begin{align*}
        \max_{k = 0, \ldots N-1}{\norm{x_{kh}}}
        &\lesssim \mf m + \sqrt{\frac{T}{\gamma} \,\bigl(\eu R_2(\mu_0\mmid \mu^{(a)}) + \log \frac{N}{\delta}\bigr)}\,.
    \end{align*}
    Finally, combining this with a union bound to control $\norm{v_{kh}}$ from Lemma \ref{lem:concentration_target_lsi}, we get the Proposition.
\end{proof}

\subsection{Completing the Discretization Proof}

We proceed by following the proof of~\cite{chewi2021analysis}.

\medskip{}

\begin{proof}[Proof of Proposition~\ref{prop:main_disc_bd}]
    Let $\{x_t\}_{t \geq 0}$ follow the continuous-time process. Let $P_T, Q_T$ denote the measures on the path space corresponding to the interpolated process and the continuous-time diffusion respectively, with both being initialized at $\mu_0 = \pi_0 \otimes \mc{N}(0, I_d)$. Then, define
    \begin{align*}
        G_t \defeq \frac{1}{\sqrt{2\gamma }} \int_0^t \langle \nabla U(x_\tau) - \nabla U(x_{\floor{\tau/h} h}), \D B_\tau \rangle - \frac{1}{4\gamma} \int_0^t \norm{\nabla U(x_\tau) - \nabla U(x_{\floor{\tau/h}h})}^2 \, \D \tau.
    \end{align*}
    
    From Girsanov's theorem (Theorem \ref{thm:girsanov}), we obtain immediately using It\^o's formula
    \begin{align*}
        \E_{Q_T} \Bigl[\Bigl(\frac{\D P_T}{\D Q_T}\Bigr)^q \Bigr] - 1
        &= \E \exp(q G_T) - 1 \\
        &= \frac{q\,(q-1)}{4\gamma} \E \int_0^T \exp(qG_t)\, \norm{\nabla U(x_t) - \nabla U(x_{\floor{t/h}h})}^2 \, \D t \\
        &\leq \frac{q^2}{4\gamma} \int_0^T \sqrt{\E[\exp(2qG_t)] \E[\norm{\nabla U(x_t) - \nabla U(x_{\floor{t/h}h})}^4]} \, \D t\,.
    \end{align*}
    Bounding these terms individually, we first use Corollary \ref{cor:girsanov} and \eqref{eq:holder} to get
    \begin{align*}
        \E \exp(2qG_t) &\leq \sqrt{\E \exp\Bigl(\frac{4q^2}{\gamma} \int_0^t \norm{\nabla U(x_r) - \nabla U(x_{\floor{r/h}h})}^2 \, \D r\Bigr)} \\
        &\leq \sqrt{\E \exp \Bigl(\frac{4L^2 q^2}{\gamma} \int_0^t \norm{x_r - x_{\floor{r/h}h}}^
        {2s} \, \D r \Bigr)}\,. 
    \end{align*}
    Let us now condition on the event
    \begin{align*}
        \mc E_{\delta, kh} \defeq \Bigl\{ \max_{j=0, 1,\dotsc, k-1}{\norm{x_{kh}}} \le R_\delta^x, \; \max_{j=0,1,\dotsc,k-1}{\norm{v_{kh}}} \le R_\delta^v\Bigr\}\,.
    \end{align*}
    By Proposition \ref{prop:high_prob_diffusion}, we can have $\P(\mc E_{\delta,kh}) \ge 1-\delta$ while choosing
    \begin{align*}
        R_\delta^x
        &\lesssim \mf m + \sqrt{\frac{\gamma}{T} \,\bigl(\eu R_2(\mu_0\mmid \mu^{(a)}) + \log \frac{N}{\delta}\bigr)}\,, \\
        R_\delta^v
        &\lesssim \sqrt d + \sqrt{\eu R_2(\mu_0\mmid \mu^{(a)}) + \log \frac{N}{\delta}}\,.
    \end{align*}
    
    We proceed to bound our desired quantity through some careful steps.
    
    \textbf{One step error.} Consider first the error on a single interval $[0,h]$. If we presume that the step size satisfies $h \lesssim (\gamma^{1-s}/(L^2 d^s q^2))^{1/(1+3s)}$, Lemma \ref{lem:better_stoc_calc} implies
    \begin{align*}
        \log \E \exp\Bigl(\frac{8L^2 q^2}{\gamma} \int_0^h \norm{x_t - x_0}^2  \, \D t\Bigr) &\le \log \E \exp \Bigl( \frac{8L^2 hq^2}{\gamma} \sup_{t \in[0,h]} {\norm{x_t - x_0}}^2 \Bigr) \\
        &\lesssim \frac{L^{2+2s} h^{1+4s} q^2}{\gamma}\, (1+\norm{x_0}^{2s^2}) + \frac{L^2 h^{1+2s} q^2}{\gamma}\, \norm{v_0}^{2s} \\
        &\qquad{} + \frac{L^2 d^s h^{1+3s} q^2}{\gamma^{1-s}}\,.
    \end{align*}
    
    \textbf{Iteration.} If we let $\{\mc F_t\}_{t \geq 0}$ denote the filtration, then writing $H_t = \int_0^t \norm{x_r - x_{\floor{r/h}h}}^2 \, \D r$, we can condition on $\mc F_{(N-1)h}$ and iterate our one step bound.
    \begin{align*}
        &\log \E \Bigl[\exp\Bigl(\frac{8L^2 q^2}{\gamma}\, H_{Nh}\Bigr)\, \ind_{\mc E_{\delta, Nh}}\Bigr] \\
        &\qquad \le \log \E \Bigl[\exp\Bigl( \frac{8L^2 q^2}{\gamma}\, H_{(N-1)h} \\
        &\qquad\qquad\qquad\qquad\qquad{} + \mc O\bigl(
        \frac{L^{2+2s} h^{1+4s} q^2}{\gamma}\, (1+\norm{x_{(N-1)h}}^{2s^2}) \\
        &\qquad\qquad\qquad\qquad\qquad\qquad\qquad{} + \frac{L^2 h^{1+2s} q^2}{\gamma}\, \norm{v_{(N-1)h}}^{2s}+\frac{L^2 d^s h^{1+3s} q^2}{\gamma^{1-s}}
        \bigr)\Bigr)
        \, \ind_{\mc E_{\delta, Nh}} \Bigr] \\
        &\qquad \le \log \E \Bigl[\exp \Bigl( \frac{8L^2 q^2}{\gamma}\, H_{(N-1)h}\Bigr)\,\ind_{\mc E_{\delta, (N-1)h}} \Bigr] \\
        &\qquad\qquad{} + \mc O\Bigl(
        \frac{L^{2+2s} h^{1+4s} q^2}{\gamma}\, (R_\delta^x)^{2s^2}
        + \frac{L^2 h^{1+2s} q^2}{\gamma}\, (R_\delta^v)^{2s}+\frac{L^2 d^s h^{1+3s} q^2}{\gamma^{1-s}}
        \Bigr)\,.
    \end{align*}
    We now make additional simplifying assumptions to obtain more interpretable bounds: we assume $\gamma/T\le 1$ and $h\lesssim \frac{1}{L} \,(1 \wedge \frac{d^{1/2}}{\mf m^s})$.
    With these assumptions,
    \begin{align*}
        &\log \E \Bigl[\exp\Bigl(\frac{8L^2 q^2}{\gamma}\, H_{Nh}\Bigr)\, \ind_{\mc E_{\delta, Nh}}\Bigr] \\
        &\qquad \le \log \E \Bigl[\exp \Bigl( \frac{8L^2 q^2}{\gamma}\, H_{(N-1)h}\Bigr)\,\ind_{\mc E_{\delta, (N-1)h}} \Bigr] \\
        &\qquad\qquad{} + \mc O\Bigl( \frac{L^2 h^{1+2s} q^2}{\gamma} \, \bigl(d + \eu R_2(\mu_0 \mmid \mu^{(a)}) + \log \frac{N}{\delta}\bigr)^s\Bigr)\,.
    \end{align*}
    Completing this iteration yields
    \begin{align*}
        \log \E \Bigl[\exp\Bigl(\frac{8L^2 q^2}{\gamma} \,H_{Nh}\Bigr)\, \ind_{\mc E_{\delta, Nh}}\Bigr] 
        &\lesssim \frac{L^2 T h^{2s} q^2}{\gamma} \, \bigl(d + \eu R_2(\mu_0 \mmid \mu^{(a)}) + \log \frac{N}{\delta}\bigr)^s\,.
    \end{align*}
    
    Finally, applying Lemma \ref{lem:uncondition} when 
    \begin{align}
        h \lesssim_s
        \frac{\gamma^{1/(2s)}}{L^{1/s} T^{1/(2s)} q^{1/s}}
    \end{align}
    (where $\lesssim_s$ hides an $s$-dependent constant), we find
    \begin{align*}
        \log \E \Bigl[\exp\Bigl(\frac{4L^2 q^2}{\gamma}\, H_{Nh}\Bigr)\Bigr]
        &\lesssim 1 + \frac{L^2 T h^{2s} q^2}{\gamma} \, \bigl(d + \eu R_2(\mu_0 \mmid \mu^{(a)}) + \log N\bigr)^s\,.
    \end{align*}
    It remains to choose the appropriate step size $h$ which makes this whole quantity $\lesssim 1$. In particular, it suffices to choose
    \begin{align}\label{eq:step_size_cond}
        h \lesssim
        \widetilde{\mc O}_s\Bigl(\frac{\gamma^{1/(2s)}}{L^{1/s} T^{1/(2s)} q^{1/s}\,(d + \eu R_2(\mu_0 \mmid \mu^{(a)}))^{1/2}}\Bigr)\,.
    \end{align}
    
    \textbf{Second term.} It remains to bound the other term in our original expression. From Lemma \ref{lem:better_stoc_calc}, we obtain 
    \begin{align*}
        \E[\exp(\lambda\, \norm{x_{kh+t} - x_{kh}}^{2s}) \mid w_{kh}] \lesssim 1\,, 
    \end{align*}
    so long as $\lambda$ is chosen to be appropriately small, i.e.,
    \begin{align*}
        \lambda \asymp \frac{1}{\gamma^s d^s h^{3s} } \wedge \frac{1}{L^{2s} h^{4s}\, (1+\norm{x_{kh}})^{2s^2}} \wedge \frac{1}{h^{2s} \, \norm{v_{kh}}^{2s}}\,.
    \end{align*}
    This immediately implies a tail bound: for $\eta \ge 0$,
    \begin{align*}
        \Pr\{\norm{x_{kh+t} - x_{kh}}^{4s} \geq \eta \mid w_{kh}\} \lesssim \exp(-\lambda \sqrt{\eta})\,.
    \end{align*}
    Integrating, we get
    \begin{align*}
        &\sqrt{\E[\norm{\nabla U(x_t) - \nabla U(x_{kh})}^4]}
        \le L^2 \sqrt{\E[\norm{x_t - x_{kh}}^{4s}]}
        \lesssim L^2 \sqrt{\E \frac{1}{\lambda^2}}
        \\
        &\qquad \lesssim L^2 \gamma^s d^{s} h^{3s} + L^{2+2s} h^{4s} \sqrt{1+\E[\norm{x_{kh}}^{4s^2}]} + L^2 h^{2s}\sqrt{\E[\norm{v_{kh}}^{4s}]}\,.
    \end{align*}
    We can estimate the expectations by integration of our previous tail bound (Proposition~\ref{prop:high_prob_diffusion}): 
    \begin{align*}
        \sqrt{\E[\norm{\nabla U(x_t) - \nabla U(x_{kh})}^4]}
        &\lesssim L^2 \gamma^s d^s h^{3s} + L^{2+2s} h^{4s}  \, \Bigl(\mf m + \frac{T}{\gamma}\,\bigl(\eu R_2(\mu_0 \mmid \mu^{(a)}) + \log N\bigr)\Bigr)^{2s^2} \\
        &\qquad{} + L^2 h^{2s}\, \bigl( d + \eu R_2(\mu_0 \mmid \mu^{(a)}) + \log N \bigr)^s \\
        &\le \widetilde{\mc O}\Bigl(L^2 h^{2s}\, \bigl(d+\eu R_2(\mu_0 \mmid \mu^{(a)})\bigr)^s\Bigr)\,,
    \end{align*}
    provided that $h \le \widetilde{\mc O}(\frac{1}{L}\,( \frac{d^{1/2}}{\mf m^s} \wedge \eu R_2 \, (\frac{\gamma \eu R_2}{T})^{s/2}))$, where $\eu R_2 = \eu R_2(\mu_0 \mmid \mu^{(a)})$. In our applications, this condition is not dominant and can be disregarded.

    \textbf{Combining the bounds.} Finally, we can combine each of these steps
    to find that, provided~\eqref{eq:step_size_cond} for the step size holds,
    \begin{align*}
        \E_{Q_T}\Bigl[\Bigl(\frac{\D P_T}{\D Q_T}\Bigr)^q\Bigr] -1
        &\le \widetilde{\mc O}\Bigl(\frac{Tq^2}{\gamma}
        \,L^2 h^{2s} \, \bigl(d + \eu R_2(\mu_0 \mmid \mu^{(a)})\bigr)^s\Bigr)\,.
     \end{align*}
    Finally, the following step size condition suffices to bound the R\'enyi divergence by $\epsilon^2$:
        \begin{align*}
        h \lesssim
       \widetilde{\mc O}_s\Bigl( \frac{\gamma^{1/(2s)} \epsilon^{1/s}}{L^{1/s} T^{1/(2s)} q^{1/s}\,(d + \eu R_2(\mu_0 \mmid \mu^{(a)}))^{1/2}} \Bigr)\,.
    \end{align*}
    This completes the proof.
\end{proof}

\section{Proof of the Main Results}\label{scn:proofs}

Firstly, we collect some results on feasible initializations from \cite{chewi2021analysis}.
Recall that $\pi^{(a)}$ is the modified distribution introduced in Appendix~\ref{scn:subgaussianty}. Let $$\pi_0 = \mc{N}(0, \varsigma I_d),$$ where $\varsigma = (2L + \beta)^{-1}$ is the variance of the Gaussian, and $\beta \asymp 1/T$ is the parameter appearing in the modified potential. The choice of $T$ will be assumption dependent, and we collect the conditions below under our main assumptions:
\begin{align*}
    T = \begin{cases}
    \widetilde{\Theta}\bigl(\frac{L+d}{\mf q(\gamma)}\bigr) & \text{$\pi$ satisfies \eqref{eq:pi}} \\
    \widetilde{\Theta}(\sqrt{L} C_{\msf{LSI}}) & \text{$\pi$ satisfies \eqref{eq:LSI}, or is strongly log-concave,}
    \end{cases}
\end{align*}
where $\mf q(\gamma)$ is defined in \eqref{eq:poincare_decay}.

\begin{lemma}[{Adapted from~\cite[Appendix A]{chewi2021analysis}}]\label{lem:initialization}
Suppose that $\pi$ satisfies~\eqref{eq:pi} and the H\"older continuity condition~\eqref{eq:holder}, as well as $\nabla U(0) = 0$, $U(0) - \min U \lesssim d$. Then the following two properties hold for $\pi_0 = \mc{N}(0, (2L + \beta)^{-1} I_d)$, where $\beta$ is the parameter appearing in the modified potential:
\begin{align*}
    \eu R_{q}(\pi_0 \mmid \pi) &\le \widetilde{\mc O}(\beta + L +  d)\,, \\
    \eu R_q(\pi_0 \mmid \pi^{(a)})& \le \widetilde{\mc O}(\beta + L + d)\,.
\end{align*}
\end{lemma}
\begin{proof}
    Apply either~\cite[Lemma 30]{chewi2021analysis} or~\cite[Lemma 31]{chewi2021analysis}.
\end{proof}

From our analysis we take $\beta \lesssim L$, and if moreover $L \lesssim d$ then it is reasonable to expect that $\eu R_q(\pi_0 \mmid \pi), \eu R_q(\pi_0 \mmid \pi^{(a)}) \le \widetilde{\mc O}(d)$.
Let $\mu_0 = \pi_0 \otimes \rho$, so that $\eu R_q(\mu_0 \mmid \mu) = \eu R_q(\pi_0 \mmid \pi)$, and similarly $\eu R_q(\mu_0 \mmid \mu^{(a)}) = \eu R_q(\pi_0 \mmid \pi^{(a)})$.

The following lemma gives a bound on the value of the Fisher information at initialization.

\begin{lemma}\label{lem:lyapunov_initialization}
    Under the conditions of the previous lemma, the initialization $\mu_0 = \pi_0\otimes \rho$ also satisfies $\msf{FI}(\mu_0 \mmid \mu) \lesssim Ld + L^{1-s} d^s$.
\end{lemma}
\begin{proof}
     Note that as $\nabla U(0) = 0$, $\norm{\nabla \log \pi(x)}^2 = \norm{\nabla U(x)}^2 \leq L^2\, \norm{x}^{2s}$. Secondly, $\pi_0$ satisfies $\E_{x \sim \pi_0}[\norm{x}^2] \lesssim d/L$.
     Hence,
    \begin{align*}
        \msf{FI}(\mu_0 \mmid \mu)
        &= \E_{\pi_0}\Bigl[\Bigl\lVert \nabla \log \frac{\pi_0}{\pi} \Bigr\rVert^2\Bigr]
        \le \E_{x \sim \pi_0}[\norm{\nabla U(x) - (2L+\beta)\, x}^2] \\
        &\lesssim L^2 \E_{x \sim \mu_0}[\norm{x}^2 + \norm x^{2s}]
        \lesssim Ld + L^{1-s} d^s\,,
    \end{align*}
    where we used Jensen's inequality in the last step.
\end{proof}

\subsection{Poincar\'e Inequality}

\begin{proof}[Proof of Theorem~\ref{thm:chip_pi}]
The continuous-time result from Lemma \ref{lem:cont_time_pi} states that
\begin{align*}
    T \gtrsim \frac{1}{\mf{q}(\gamma)} \log\frac{\chi_2(\mu_0 \mmid \mu)}{\varepsilon^2} \implies \chi_2(\mu_T \mmid \mu) \leq \epsilon^2\,.
\end{align*}
Noting that there exists a feasible initialization such that $\log \chi_2(\mu_0 \mmid \mu) \le \widetilde{\mc O}(L + d)$, then this is satisfied if we choose $T = \widetilde {\mc{O}}(\frac{1}{\mf q(\gamma)}\,(L + d + \log\frac{1}{\varepsilon}))$. This also shows that $\eu R_2(\mu_T \mmid \mu) = \log(1 + \chi_2(\mu_T \mmid \mu)) \lesssim \epsilon^2$ for $\epsilon \lesssim 1$. 

Note the following decomposition (weak triangle inequality) for the R\'enyi divergence~\cite[see, e.g.,][Proposition 11]{mironov2017renyi}:
\begin{align*}
    \eu R_q(P_1\mmid P_2) \leq \frac{q-1/\mf c}{q-1}\,\eu R_{\mf c q}(P_1 \mmid P_3) + \eu R_{\mf d(q- 1/\mf c)}(P_3 \mmid P_2),
\end{align*}
for any valid H\"older conjugate pair $\mf c, \mf d$, i.e., $\frac{1}{\mf c} + \frac{1}{\mf d} = 1$, $\mf c, \mf d > 1$, and any three probability distributions $P_1, P_2, P_3$.

In our case, we let $q = 2-\xi$ and $\mf d (q - 1/\mf c) = 2$, so that after solving for $\mf c, \mf d$, we get the following for $\xi \leq 1/2$:
\begin{align*}
    \eu R_{2-\xi}(P_1 \mmid P_2) &\leq 2\eu R_{2/\xi}(P_1 \mmid P_3) + \eu R_2(P_3 \mmid P_2)\,.
\end{align*}
Consequently, let $P_1 = \hat \mu_{Nh}$, $P_2 = \mu$, $P_3 = \mu_{Nh}$, and combining this result with the discretization bound of Proposition \ref{prop:main_disc_bd}, we then obtain
\begin{align*}
    \eu R_{2-\xi}(\hat \mu_{Nh} \mmid \mu) \lesssim \eu R_{2/\xi}(\hat \mu_{Nh} \mmid \mu_{Nh}) + \eu R_2(\mu_{Nh} \mmid \mu) \lesssim \epsilon^2\,,
\end{align*}
so long as
\begin{align*}
    h &= \widetilde{\Theta} \Bigl(\frac{\gamma^{1/(2s)} \epsilon^{1/s} \xi^{1/s} \mf q(\gamma)^{1/(2s)}}{L^{1/s} d^{1/2} \,{(L \vee d)}^{1/(2s)}} \Bigr)\,,  \\
    N &= \widetilde{\Theta} \Bigl(\frac{L^{1/s} d^{1/2} \, {(L\vee d)}^{1+1/(2s)}}{\gamma^{1/(2s)} \epsilon^{1/s}  \xi^{1/s} \mf q(\gamma)^{1+1/(2s)}} \Bigr)\,.
\end{align*}
This completes the proof.
\end{proof}

\subsection{Log-Sobolev Inequality}

\subsubsection{KL Divergence}

\begin{proof}[Proof of Theorem~\ref{thm:kl_slc}]
    We provide the following Theorem in the twisted coordinates $(\phi, \psi)$, which were used in Lemma \ref{lem:lsi_ct}. Consider the decomposition of the $\msf{KL}$ using Cauchy--Schwarz:
    \begin{align*}
        \msf{KL}(\hat \mu_T^{\mc M} \mmid \mu^{\mc M}) &= \int \log \frac{\hat \mu_T^{\mc M}}{\mu^{\mc M}} \, \D \hat \mu_T^{\mc M} \\
        &= \msf{KL}(\hat \mu_T^{\mc M} \mmid \mu_T^{\mc M}) + \int \log \frac{\mu_T^{\mc M}}{\mu^{\mc M}} \, \D \hat \mu_T^{\mc M} \\
        &= \msf{KL}(\hat \mu_T^{\mc M} \mmid \mu_T^{\mc M}) + \msf{KL}(\mu_T^{\mc M} \mmid \mu^{\mc M}) + \int \log \frac{\mu_T^{\mc M}}{\mu^{\mc M}} \, \D(\hat \mu_T^{\mc M} - \mu_T^{\mc M}) \\
        &\leq \msf{KL}(\hat \mu_T^{\mc M} \mmid \mu_T^{\mc M}) 
 + \msf{KL}(\mu_T^{\mc M} \mmid \mu^{\mc M}) + \sqrt{\chi^2(\hat \mu_T^{\mc M} \mmid \mu_T^{\mc M})\times \msf{var}_{\mu_T^{\mc M}}\Bigl(\log \frac{\mu_T^{\mc M}}{\mu^{\mc M}}\Bigr)}\,.
    \end{align*}
    Using the log-Sobolev inequality of the iterates via Lemma~\ref{lem:lsi_ct}, we find (through the implication that a log-Sobolev inequality implies a Poincar\'e inequality with the same constant)
    \begin{align*}
        \msf{var}_{\mu_T^{\mc M}}\Bigl( \log \frac{\mu_T^{\mc M}}{\mu^{\mc M}}\Bigr)  \leq C_{\msf{LSI}}(\mu_T^{\mc M}) \E_{\mu_T^{\mc M}}\Bigl[\Bigl\lVert\nabla \log \frac{\mu_T^{\mc M}}{\mu^{\mc M}}\Bigr\rVert^2\Bigr]\,,
    \end{align*}
    where we substitute $\log \frac{\mu_T^{\mc M}}{\mu^{\mc M}}$ for the function in \eqref{eq:pi}.
    Here, $C_{\msf{LSI}}(\mu_T^{\mc M})\lesssim 1/m$ for all $t\ge 0$.

    Since $\mu^{\mc M} = \mc M_\# \mu$, then $\mu^{\mc M}(\phi,\psi) \propto \mu(\mc M^{-1}(\phi,\psi))$.
    Therefore,
    \begin{align*}
        \nabla \log \mu^{\mc M} = (\mc M^{-1})^\T\, \nabla \log \mu \circ \mc M^{-1}\,,
    \end{align*}
    and similarly for $\nabla \log \mu_T^{\mc M}$.
    This yields the expression
    \begin{align*}
        \E_{\mu_T^{\mc M}}\Bigl[\Bigl\lVert\nabla \log \frac{\mu_T^{\mc M}}{\mu^{\mc M}}\Bigr\rVert^2\Bigr]
        &= \E_{\mu_T}\Bigl[\Bigl\lVert (\mc M^{-1})^\T \,\nabla \log \frac{\mu_T}{\mu} \Bigr\rVert^2\Bigr]\,.
    \end{align*}
    Also, one has
    \begin{align*}
        \mc M^{-1}\,(\mc M^{-1})^\T
        &= \begin{bmatrix}
            1 & -\gamma/2 \\
            -\gamma/2 & \gamma^2/2
        \end{bmatrix}\,.
    \end{align*}
    For $c_0 > 0$ and $\mf M$ defined in Appendix~\ref{scn:entropic_hypo}, we have
    \begin{align*}
        L\,\mf M - c_0\,\mc M^{-1} \, (\mc M^{-1})^\T
        &= \begin{bmatrix}
            1/4-c_0 & \sqrt L\,(1/\sqrt 2 + c_0\sqrt 2) \\
            \sqrt L\,(1/\sqrt 2 + c_0\sqrt 2) & L\,(4-c_0)
        \end{bmatrix}\,.
    \end{align*}
    The determinant is $L\, ((\frac{1}{4}-c_0)\,(4-c_0) - (\frac{1}{\sqrt 2} + c_0\sqrt 2)^2) > 0$ for $c_0 > 0$ sufficiently small.
    This shows that $\mc M^{-1}\,(\mc M^{-1})^\T \preceq c_0^{-1} L\,\mf M$, and therefore
    \begin{align*}
        \E_{\mu_T^{\mc M}}\Bigl[\Bigl\lVert\nabla \log \frac{\mu_T^{\mc M}}{\mu^{\mc M}}\Bigr\rVert^2\Bigr]
        &\lesssim L\,\msf{FI}_{\mf M}(\mu_T \mmid \mu)\,.
    \end{align*}
    Here we define
    \begin{align*}
        \msf{FI}_{\mf M}(\mu' \mmid \mu) \defeq \E_{\mu'}\bigl[\bigl\lVert \mf M^{1/2}\, \nabla \log \frac{\mu'}{\mu} \bigr\rVert^2 \bigr]
    \end{align*}
    
     The decay of the Fisher information via Lemma~\ref{lem:lyapunov_decay} allows us to set
     \begin{align*}
        T \gtrsim C_{\msf{LSI}}\sqrt L \log \Bigl(\frac{\kappa}{\epsilon^2}\,\bigl(\msf{KL}(\mu_0 \mmid \mu) + \msf{FI}_{\mf M}(\mu_0 \mmid \mu)\bigr)\Bigr) \implies \msf{var}_{\mu_T^{\mc M}}\Bigl( \log \frac{\mu_T^{\mc M}}{\mu^{\mc M}}\Bigr) \lesssim \epsilon^2\,.
     \end{align*}
     The same choice of $T$ also ensures that $\msf{KL}(\mu_T^{\mc M} \mmid \mu^{\mc M}) \le \varepsilon^2$.
     From our initialization (Lemma \ref{lem:lyapunov_initialization}), we can naively estimate using that
     \begin{align*}
         \msf{FI}_{\mf M}(\mu_0 \mmid \mu)
         &\lesssim \frac{1}{L} \, \msf{FI}(\pi_0 \mmid \pi)
         \lesssim d\,,
     \end{align*}
     and $\msf{KL}(\mu_0 \mmid \mu) \lesssim d \log \kappa$, so that
     our condition on $T$ is (with $C_{\msf{LSI}} \le m^{-1}$)
     \begin{align*}
         T \ge \widetilde{\mc O}\Bigl(\frac{\sqrt{L}}{m} \log \frac{\kappa d}{\epsilon^2}\Bigr)\,.
     \end{align*}
    Recall as well that this requires $\gamma \asymp \sqrt{L}$.
     For the remaining $\chi^2(\hat \mu_T \mmid \mu_T)$ and $\msf{KL}(\hat \mu_T \mmid \mu_T)$ terms, we invoke Proposition \ref{prop:main_disc_bd} with the value of $T = Nh$ specified and desired accuracy $\epsilon$, and with $q = 2$ and $s=1$, which consequently yields
     \begin{align*}
         h = \widetilde{\Theta} \Bigl(\frac{\epsilon m^{1/2}}{Ld^{1/2}} \Bigr)\,,
     \end{align*}
    with 
      \begin{align*}
        N = \widetilde{\Theta} \Bigl(\frac{\kappa^{3/2} d^{1/2}}{\epsilon} \Bigr)
     \end{align*}
     (using $N = T/h$).
\end{proof}

\subsubsection{TV Distance}

\begin{proof}[Proof of Theorem~\ref{thm:tv_lsi}]
Notice first that the $\msf{TV}$ distance is a proper metric, and therefore satisfies the triangle inequality. Subsequently, by two applications of Pinsker's inequality,
\begin{align*}
    \norm{\hat \mu_{Nh} - \mu}_{\msf{TV}} &\leq \norm{\hat \mu_{Nh} - \mu_{Nh}}_{\msf{TV}} + \norm{\mu_{Nh} - \mu}_{\msf{TV}} \\
    &\lesssim \sqrt{\msf{KL}(\hat \mu_{Nh} \mmid \mu_{Nh})} + \sqrt{\msf{KL}(\mu_{Nh} \mmid \mu)}\,.
\end{align*}

These terms can be bounded separately. Analogous to the proof of the prior theorem, using Lemma \ref{lem:lyapunov_decay}, it suffices to take
     \begin{align*}
         T \ge \widetilde{\mc O}\Bigl(C_{\msf{LSI}}\sqrt L \log \frac{d}{\epsilon^2}\Bigr)\,,
     \end{align*}
and for the other term, it suffices to use Proposition \ref{prop:main_disc_bd} with any value of $q$, $\gamma \asymp \sqrt{L}$ which combined with the requirement on $T$ yields:
\begin{align*}
         h = \widetilde{\Theta} \Bigl(\frac{\epsilon}{C_{\msf{LSI}}^{1/2} L d^{1/2}} \Bigr)\,,
     \end{align*}
     with 
\begin{align*}
        N = \widetilde{\Theta} \Bigl(\frac{C_{\msf{LSI}}^{3/2} L^{3/2} d^{1/2}}{\epsilon} \Bigr)\,,
\end{align*}
(using $N = T/h$).
\end{proof}

\end{document}